\documentclass{conm-p-l}
\usepackage{amssymb}
\usepackage{amsthm}
\usepackage{color}
\usepackage{graphicx}
\usepackage{latexsym}
\usepackage{hyperref}
\usepackage{dsfont}

\newtheorem{theorem}{Theorem}[section]
\newtheorem{lemma}[theorem]{Lemma}

\newtheorem{proposition}[theorem]{Proposition}

\newtheorem{corollary}[theorem]{Corollary}

\theoremstyle{definition}
\newtheorem{definition}[theorem]{Definition}
\newtheorem{example}[theorem]{Example}

\theoremstyle{remark}

\newtheorem{openproblem}[theorem]{Open Problem}

\numberwithin{equation}{section}

\newcommand{\abs}[1]{\lvert#1\rvert}

\newcommand{\F}{\mathcal{F}}

\newcommand{\G}{\mathcal{G}}

\newcommand{\Z}{\mathbb{Z}}
\newcommand{\R}{\mathds{R}}
\newcommand{\N}{\mathbb{N}}

\newcommand{\Pp}{\mathcal{P}}
\newcommand{\real}{\mathds{R}}
\newcommand{\rr}{\mathds{R}}

\newcommand{\vi}{{\sf V}}

\DeclareMathOperator{\vol}{vol}

\DeclareMathOperator{\conv}{conv}
\DeclareMathOperator{\diam}{diam}

\def\N{\mathbb{N}}

\def\R{\mathds{R}}
\def\Z{\mathbb{Z}}

\newcommand{\D}{\mathcal{D}}
\newcommand{\B}{\mathcal{B}}
\newcommand{\HH}{\mathcal{H}}

\begin{document}

\title[Helly Today]{Helly's Theorem: New Variations and Applications}

\author{Nina Amenta}
\address{Department of Computer Science, University of California, Davis, One Shields Avenue, Davis, California 95616}
\email{amenta@cs.ucdavis.edu}

\author{Jes\'us~A.~De~Loera}
\address{Department of Mathematics, University of California, Davis, One Shields Avenue, Davis, California 95616}
\email{deloera@math.ucdavis.edu}

\author{Pablo Sober\'on}
\address{Department of Mathematics, Northeastern University, 360 Huntington Ave., Boston, Massachusetts 02115}
\email{p.soberonbravo@neu.edu}
\thanks{}

\subjclass[2010]{ 52B05}

\date{}

\begin{abstract} This survey presents recent Helly-type geometric theorems published since the appearance 
of the last comprehensive survey, more than ten years ago. We discuss how such theorems continue to be 
influential in computational geometry and in optimization. The survey contains several open problems.
\end{abstract}

\maketitle

\section{Introduction}

Eduard Helly's theorem is now a century old\footnote{The theorem was stated by Eduard Helly already in 1913, but his participation in the first world war delayed the publication of his proof until 1923. Alas, Johann Radon published a now classical proof a few years earlier too.} \cite{Helly:1923wr, Radon:1921vh} and it is firmly established as a fundamental result in combinatorial geometry, one with a large number of applications. In its original form it states.

\begin{theorem}[Helly's theorem]
	Let $\F$ be a finite family of convex sets in $\R^d$.  If every $d+1$ or fewer elements of $\F$ intersects, then all the sets in family $\F$ intersect.
\end{theorem}

 We recommend \cite{Matousek:2002td} for an in-depth introduction to combinatorial geometry.
Since its discovery, Helly's theorem has found a plethora of generalizations, extensions and applications in many areas of mathematics
(see  \cite{Danzer:1963ug,Eckhoff:1993uy, Matousek:2002td, Wenger:2004uf} and the references therein). Our survey  focuses on recent advances, roughly after the year 2004, 
regarding new versions of Helly's theorem and its generalizations. We also discuss applications of Helly theorems in computational geometry and optimization.
The developments around this area have been tremendous. 

\section{The many variations and manifestations of Helly's theorem}

One of the reasons for the popularity of Helly's theorem is its versatility; there are many ways to change the 
framework that yield interesting theorems.  Changing the following aspects of Helly's theorem are common ways to get different versions.
\begin{itemize}
	\item \textbf{Convexity of the sets.}  The convexity hypothesis on the sets in
	$\F$ can be replaced by either stronger conditions, such as being axis-parallel boxes or translates of a convex set, 
	or weaker conditions such as conditions on the homology groups of the sets and their intersections.
	
	\item \textbf{Quality of intersection as a conclusion.}  One can replace the goal of having the whole family intersect by other conditions.  This could include containing some integer points, having an intersection of a certain dimension or volume, having a bounded piercing number, etc.
	
	\item \textbf{Local intersection conditions.}  Instead of requiring that all subfamilies of a certain size intersect, one can ask for other combinatorial conditions on families of a fixed size.  Some examples are the $(p,q)$ property, fractional conditions, or colorful conditions.
\end{itemize}

Given a property $\Pp$, it is standard to call $\Pp$ a \textit{Helly-type property} if there is an integer $h$ such that for a family $\F$ of sets, \textit{if every $h$ elements of $\F$ satisfy $\Pp$, so does the entire family $\F$}.  The smallest such integer $h(\Pp)$ is the Helly number of property $\Pp$. Helly's theorem essentially says that \textit{having non-empty intersection} is a Helly-type property for all finite families of convex sets in $\R^d$, with $d+1$ as the Helly number.  In this survey we restrict ourselves to Helly-type properties related to geometric intersection of sets, sometimes with additional conditions.
In this context, there is a natural way to define Helly numbers for a given family $\F$ of sets, bypassing the specification of a property.
Given a family $\F$ of sets, we define the $\F$-helly number $h(\F)$ (if it exists) as the smallest integer satisfying the following.  For any finite subfamily $\G\subset \F$, if every $h(\F)$ sets of $\G$ intersect, all of the sets in $\G$ intersect too.	
If $h(\G)$ is undefined, we simply say $h(\G) = \infty$.

With more than 300 papers published in the last fifteen years with relation to Helly's theorem it is not a surprise that we were 
forced to restrict ourselves to Helly-type theorems related to geometric intersection of sets.  Sadly, we left detailed discussion of many 
interesting geometric Helly-type result out of this survey 
(see e.g.,  \cite{1gen,2gen,3gen,4gen,5gen,6gen,7gen,8gen,9gen,10gen,11gen,12gen,13gen,14gen,15gen,16gen} and the references therein) and we omitted other very active areas,
with less emphasis in geometry, such as Helly theorems over graphs. 

In each of the following subsections we will discuss a different family of new variants of Helly's theorem related to the intersections of 
geometric sets.  
They include coordinates restrictions and colorful results, as well as quantitative, topological, and fractional versions.  
We conclude the section with results regarding transversals to convex sets.  Many results fall into more than one category, 
so we have made some arbitrary choices in the presentation.
We include a few short proofs to help provide insight, and we establish some new results generalizing
or combining existing theorems as we go along.

\subsection{Coordinates restrictions and S-Helly numbers} \label{Shelly}

Let us motivate what $S$-Helly numbers are with an example.

\begin{theorem}[Boxes-only Helly Theorem; Halman \cite{Halman:2007gu}]\label{theorem-boxes-helly}
Given a finite set $D$ of at least $2d$ axis parallel boxes in $\R^d$, and a finite set $S$ of points, the boxes in $D$ have a common point in $S$ if every $2d$ of its elements do.
\end{theorem}

The result above is highly dependent on the set of points $S$.  Notice we get a radical simplification when $S= \R^d$,  namely,

\begin{theorem}\label{theorem-simple-boxes}\label{theorem-boxes-helly-simple}
Given a finite set $D$ of axis parallel boxes in $\R^d$, the boxes in $D$ have a common point if every two of its elements do.
\end{theorem}

There is a wide collection of Helly-type results, where, as in Theorem \ref{theorem-boxes-helly}, 
the hypotheses and conclusions of intersection are required to contain points from a specific proper subset $S$ of $\R^d$.  
A classic result of this sort is Doignon's theorem, in which the point set $S$ is more strictly specified to be $S=\Z^d$. The theorem
is a variation of Helly's theorem.

\begin{theorem}[Doignon \cite{Doignon:1973ht}] \label{doignon}
Given a finite family $D$ of at least $2^d$ convex sets in $\R^d$, the sets in $D$ have a common point in $\Z^d$ if every $2^d$ of its elements do.
\end{theorem}

Doignon's theorem was rediscovered by David E. Bell \cite{Bell:1977tm} and Herbert Scarf \cite{Scarf:1977va}, with applications in integer optimization in mind.
%


We can now define in general the $S$-Helly numbers:

\begin{definition}\label{definition-S-Helly}
	Given a set $S \subset \R^d$, we consider \textit{the $S$-Helly number} 
	\[
	h(S)= h(\{K \cap S: K \subset \R^d \ \mbox{is convex}\}).
	\]  
	In other words, it is the smallest positive integer such that the following statement holds.  
	\textit{Given a finite family $\F$ of convex sets in $\R^d$, if the intersection of every $h(S)$ contains a point of $S$, then there is a point of $S$ in $\cap \F$.} Note that $h(S)$ is the Helly number for the property \textit{``having non-empty intersection over $S$''}.
If no number satisfies the condition above, we say $h(S) = \infty$.  
We use $h(S)$ to denote the Helly number for set $S$, since in this context $\F$ is always the family of convex sets. 
\end{definition}

Helly's original theorem can thus be stated as $h(\R^d) = d+1$, and Doignon's theorem as $h(\Z^d) = 2^d$.
Interestingly, many other sets $S$ give Helly-type theorems as well.  For the case of  dimension $d=1$ the reader can prove easily that for any set $S\subset \R$,  the Helly number $h(S)$ exists and is at most two, but starting with $\R^2$ the situation becomes 
considerably more complicated (see for example the results in \cite{de2015helly}).

Historically, many results about $S$-Helly numbers were originally discussed in the setting of \emph{abstract convexity}, which
focuses on the study of spaces emulating a \textit{convex hull operator} and recreating many results which were known only in 
Euclidean spaces \cite{convexityspaces-vandevel}.  
Many \textit{convexity spaces} can be formed by taking a proper subset $S \subset \R^d$, and defining the 
convex sets as the intersections of the standard convex sets in $\R^d$ with $S$, 
as in Definition \ref{definition-S-Helly} (see \cite{Doignon:1981fv,Hoffman:1979ix,Jamison:1981wz, Kay:1971uf} and the references therein).    

In recent years, the excitement around $S$-Helly numbers has increased due to the rich connections to the theory of optimization, in particular
cutting planes (see the end of Subsection \ref{soptimization}).  Motivated by the work of  Bell, and Scarf, A. J. Hoffman developed a theory to compute general $S$-Helly numbers in \cite{Hoffman:1979ix}.  Hoffmann stated and sketched the following \emph{mixed integer} version of Helly's and Doignon's theorems, which includes them both. Later it was rediscovered by Averkov and Weismantel \cite{Averkov:2010tva}, who gave a detailed proof:

\begin{theorem}[Mixed-Integer Helly theorem \cite{Hoffman:1979ix, Averkov:2010tva}] \label{mixedint}
		The Helly number $h(\R^d \times \Z^k)$ equals $(d+1)2^k$. 
\end{theorem}

The following methods are useful to bound the $S$-Helly number of a subset $S$ of $\R^d$. 
The techniques were studied first by Hoffman \cite{Hoffman:1979ix}, and later developed by several authors
 \cite{Averkov:2010tva, Ave2013,conforti-unp,de2015helly}.

We define an \emph{$S$-vertex-polytope} as the convex hull of points $x_1,x_2,\dots,x_k\in S$ in convex position such that no other point of $S$ is in $\conv\{x_1,\dots,x_k\}$. Similarly, an \emph{$S$-face-polytope} is defined as the intersection of half-spaces $H_1,H_2,\dots,H_k$ such that $\bigcap_i H_i$ has $k$ facets and contains exactly $k$ points of $S$, one contained in the relative interior of each facet.

%

A set $S \subset \R^d$ is said to be \textit{discrete} if every point of $\R^d$ has an open neighborhood that contains only a finite number of points of $S$.

\begin{lemma}[{\cite[Thm. 2.1]{Ave2013}} {\cite[Prop. 3]{Hoffman:1979ix}}]\label{lem:hollow}
Assume $S\subset\R^d$ is discrete, then $h(S)$ is equal to the following two numbers, which must be equal:
\begin{enumerate}
\item The supremum of the number of faces of an $S$-face-polytope.
\item The supremum of the number of vertices of an $S$-vertex-polytope.
\end{enumerate}
\end{lemma}

%
%

The case when $S$ is not discrete is a bit more delicate.  If we denote by $f(S)$ the quantity (1) above, Averkov \cite{Ave2013}  showed that in general $f(S) \leq h(S)$.
Recently, new results involving these two quantities and the $S$-Helly number $h(S)$ have been obtained by Conforti and Di Summa \cite{conforti-unp, Ave2013}. For example, 
they proved that when $S$ is a closed subset and $f(S)$ is finite, then $h(S) \leq (d+1) f(S)$. They also showed an example where $f(S)=1$, yet the $S$-Helly number is infinite (examples appeared in \cite{de2015helly}).  Fortunately, for the calculation of Helly numbers there is a very general result due to A.J. Hoffman characterizing $S$-Helly numbers.

\begin{lemma}[Proposition 2 in \cite{Hoffman:1979ix}]\label{lem:hof}
If $S\subset\R^d$, then $h(S)$ can be computed as the supremum of the numbers $h$ such that the following holds:
there exists a set $R=\{x_1,\dots,x_h\}\subset S$ such that 
$\bigcap_i\conv(R\setminus\{x_i\})$ does not intersect $S$.
\end{lemma}

Note that, in particular, the points of $R$ must be in strictly convex position. As a direct application of Lemma \ref{lem:hof} we have the following proposition.

\begin{proposition}
If $S_1,S_2\subseteq\R^d$, then $h(S_1\cup S_2)\leq h(S_1)+h(S_2)$.
\end{proposition}

The Helly number of a product is not always so well-behaved as in the mixed integer case, but Averkov and Weismantel  proved the following

\begin{proposition}[{\cite[Thm. 1.1]{Averkov:2010tva}}]
If $M$ is a closed subset of $\R^k$ then $h(\R^d\times M)\le (d+1)h(M)$.
\end{proposition}

Similarly, one can prove a general bound in the case of discrete sets.  The result below shows the optimality of Doignon's theorem, for example.

 \begin{theorem}[Conforti, Di Summa \cite{conforti-unp}]
 	If $S_1, S_2 \subset \rr^d$ are discrete sets, then $h(S_1 \times S_2) \ge h(S_1)h(S_2)$. 
 \end{theorem}  

It is also natural to study subgroups of $\R^d$. General subgroups and other 
new examples of sets $S$ with $h(S)< \infty$ are presented in \cite{de2015helly}.

%
%
%

\begin{theorem}[De Loera, La Haye, Oliveros, Rold\'an Pensado \cite{de2015helly}]
The following bounds on Helly numbers hold
\begin{itemize}
	\item If $S$ is a dense subset of $\R^2$, then $h(S)\leq 4$. This result is sharp.
	\item Let $L$ be a proper 
sublattice of $\Z^2$. If $S=\Z^2\setminus L$, then $h(S)\leq 6$. This result is sharp.
	\item Let $L_1,\dots,L_k$ be (possibly translated) sublattices of $\Z^d$. Then the set $S =\Z^d \setminus (L_1\cup\dots\cup L_k)$ has Helly number $h(S)\le C_k2^d$ for some constant $C_k$ depending only on $k$.
\end{itemize}	
\end{theorem}

 \subsection{Colorful variations}
 
Colorful variations of Helly's theorem 
come from asking additional combinatorial restrictions to the intersection structure of the family.  This result is due to Lov\'asz, and his proof, below,  appeared in a paper by B\'ar\'any.

\begin{theorem}[Colorful Helly theorem; Lov\'asz \cite{Barany:1982va}]
Let $\F_1, \F_2, \ldots, \F_{d+1}$ be finite families of convex sets in $\rr^d$ such that for every choice $K_1 \in \F_1, K_2 \in \F_2, \ldots, K_{d+1} \in \F_{d+1}$, the family $\{K_1, \ldots, K_{d+1}\}$ has non-empty intersection.  Then, there is an index $i$ such that the members of family $\F_i$ have a non-empty intersection.
\end{theorem}

Note that if $\F_1 = \F_2 = \ldots = \F_{d+1}$ we obtain Helly's original theorem.  The sets $\F_i$ are called \textit{color classes}, which gives the theorem its name.

\begin{proof}
We may assume without loss of generality that the sets are compact.  Let $v$ 
be a direction in $\rr^d$.  We may choose $v$ so that for every colorful choice of $d$ sets, their intersection attains its $v$-directional minimum at a single point (a consequence of \cite{Ewald:1970vw}, for instance).

Take a colorful choice of $d$ sets such that its $v$-directional minimum $p$ is maximal (in the direction $v$).  We may assume without loss of generality that the missing color in that $d$-tuple is $\F_{d+1}$, and denote its sets by $F_1 \in \F_1, \ldots, F_d \in \F_d$.  We will show that $p \in \cap \F_{d+1}$.

Take any set $F_{d+1} \in \F_{d+1}$.  By the condition of the problem, the $(d+1)$-tuple $A=\{F_1, \ldots, F_{d+1}\}$ must be intersecting.  Let $q$ be a point in its intersection.  Notice that $\langle q , v \rangle \ge \langle p , v \rangle$, where $\langle \cdot, \cdot \rangle$ represents the dot product.  Let $A_i = A \setminus \{F_i\}$ for $1 \le i \le d+1$.

By the definition of $p$, for each $i$, the $v$-directional minimum $q_i$ of $\cap A_i$ satisfies $\langle q_i, v \rangle \le \langle p, v\rangle$.  Since $\cap A_i$ also contains $q$, by convexity there must be a point $p_i \in \cap A_i$ with $\langle p_i, v\rangle = \langle p, v \rangle$ (notice that $p_{d+1} = p$).

Thus, we have obtained $d+1$ points in the hyperplane $H = \{x : \langle x, v\rangle = \langle p, v \rangle\}$, of dimension $d-1$.  By Radon's lemma \cite{Radon:1921vh}, there must be a partition of them into two sets $B$, $C$ such that $\conv B \cap \conv C \neq \emptyset$.  Take any point $p' \in \conv B \cap \conv C$.  It is immediate that $p' \in F_i$ for all $1 \le i \le d+1$.  Thus $p' = p$, which means $p \in F_{d+1}$, as desired.
\end{proof}

Not every Helly-type theorem admits a colorful version.  For instance, a colorful version of Theorem \ref{theorem-simple-boxes} does not hold even if we permit $d$ colors.  A simple 
counterexample is to take $\F$ to be the set the of facets of a hypercube.  We color $\F$ so that each opposite pair of facets carry the same color.  Clearly no color class intersects, but every colorful choice intersects in one vertex of the hypercube.  There are two interesting questions regarding colorful variations:

\begin{itemize}
	\item Which Helly-type theorems admit a colorful version?
	\item Is it possible to get a stronger conclusion than a single color class intersecting?
\end{itemize}

Sufficient conditions for the existence of colorful version were determined in \cite{de2015helly}, which are essentially the minimal conditions needed to carry out Lov\'asz's original proof.  In particular, it follows that any discrete set $S$ with a finite Helly number $h(S)$ admits a Helly-type theorem with $h(S)$ color classes.  This extends to some quantitative versions, described in the next subsection \cite{DeLoera:2015wp}.  In particular, it shows that Doignon's theorem is colorable.

\begin{theorem}[Colorful Doignon; De Loera, La Haye, Oliveros, Rold\'an-Pensado \cite{de2015helly}]
	Let $\F_1, \F_2, \ldots, F_{2^d}$ be finite families of convex sets in $\R^d$.  If the intersection of every colorful choice $F_1 \in \F_1, \ldots, F_{2^d} \in \F_{2^d}$ contains an integer point, there is a color class $\F_i$ such that $\cap \F_i$ contains an integer point.
\end{theorem}

There are still plenty of Helly-type theorems that do not fit the criteria in \cite{de2015helly}.  For instance, Karasev presented the following theorem.

\begin{theorem}[Karasev \cite{Karasev:2000ui}] \label{theorem-karasev-translates}
 Let $K$ be a convex set in the plane and $\F$ be a family of translates of $K$. If every pair of sets in $\F$ has a non-empty
intersection, then there are three points such that every set in $\F$ contains at least one of them.
\end{theorem}

Dol'nikov asked if Theorem \ref{theorem-karasev-translates} could be colored.

\begin{openproblem}[Dol'nikov's problem]
Let $K$ be a convex set in the plane and $\F_1, \F_2, \F_3$ be finite families of translates of $K$.  If any two translates of $K$ from different families intersect, show that there is an index $i$ and three points in $\R^d$ such that every set of $\F_i$ contains at least one of the points.	
\end{openproblem}

The problem above is only solved when $K$ is a triangle, $K$ is centrally symmetric, or when the hypothesis involves four families instead of three \cite{chuymagazinovsoberon}.

A stronger form of the colorful Helly theorem was shown by Kalai and Meshulam.  Their methods work in much more general settings, where the conditions on the color classes are replaced by an arbitrary matroid.  In its most direct application to convex sets in $\R^d$, it yields the theorem below.  Even though the statement is indeed combinatorially stronger than the original colored version by Lov\'asz, the reader may notice that the proof at the beginning of this subsection actually implies both results. See also the interesting commutative algebra reformulation of colorful variants in \cite{gunnar}.

\begin{theorem}[Kalai, Meshulam \cite{Kalai:2005cm}]\label{theorem-stronger?-colorful-kalai}
	Let $\F_1, \F_2, \ldots, \F_{d+1}$ be finite families of convex sets in $\rr^d$ such that for every choice $K_1 \in \F_1, K_2 \in \F_2, \ldots, K_{d+1} \in \F_{d+1}$, the family $\{K_1, \ldots, K_{d+1}\}$ has non-empty intersection.  Then, there is an index $i$ and a choice of sets $F_j \in \F_j$ for all $j \neq i$ such that the family $\F_i \cup \{F_j : j \neq i\}$ has non-empty intersection.
\end{theorem}

If we change the ambient space, colorful results also appear.  A subset $K$ of the sphere $S^{d-1}$ is called \textit{strongly convex} if it is contained in an open hemisphere and for every pair of points $x$, $y$ in $K$, the shortest arc from $x$ to $y$ is contained in $K$.  In the sphere there are different notions of convexity, which lead to different Helly-type results (see the survey \cite{Wenger:2004uf} and the references therein for details).

\begin{theorem}[Holmsen, Pach, Tverberg \cite{Holmsen:2008id}]\label{theorem-sphercial-colorful-helly}
	Let $C_1, C_2, \ldots, C_{d+1}$ be non-empty, finite collections of closed strongly convex sets on $S^{d-1}$.  If every selection of sets from distinct $C_i$ have a point in common, then for some $1 \le i < j \le d+1$ there is a point common to every member of $C_i \cup C_j$. 
\end{theorem}

This is a consequence of applying duality to the following ``very colorful Carath\'eodory'', which also appears in \cite{Arocha:2009ft}. For more on the interesting topic of colorful Carath\'eodory theorems 
see \cite{Barany:1982va, barany1995caratheodory,Barany:1997fr}

\begin{theorem}[Very colorful Carath\'eodory, \cite{Holmsen:2008id, Arocha:2009ft}]
	Let $C_1, C_2, \ldots, C_{d+1}$ be sets of points in $\R^d$, considered as color classes.  If every colorful choice $c_1 \in C_1, \ldots, c_{d+1} \in C_{d+1}$ does not capture the origin (i.e. $0 \not\in \conv\{c_1, \ldots, c_{d+1}\}$), then there are two color classes $C_i, C_j$ such that $0 \not\in \conv (C_i \cup C_j)$.
\end{theorem}

  A similar version of Theorem \ref{theorem-sphercial-colorful-helly} holds in $\R^d$ with more generality.  Namely, the following new theorem,

\begin{theorem}
	Let $n > d$ and $\F_1, \ldots, \F_n$ be non-empty, finite collections of closed convex sets in $\R^d$.  If every selection of sets, one from each $\F_i$, have a point in common, then there is a set $I$ of $n-d$ indices in $\{1, \ldots, n\}$ such that there is a point in common to every member of $\cup_{i \in I} \F_i$.
\end{theorem}

The reader may notice that the original proof by Lov\'asz, described at the beginning of this subsection, proves this result.


\subsection{Quantitative variations}

More variations of Helly's theorem appear when one wants to quantify the size of the intersection of the sets in the family. The driving principle behind these results is that \textit{given a family $\F$ of convex sets in $\R^d$, if every subfamily of a fixed size has a ``large intersection'', then the whole family has a ``large intersection'' too.}

If we consider $\mathcal{C} (\R^d)$ the space of all convex sets of $\R^d$ with the Hausdorff topology, we can measure the size of a convex set by any non-negative monotone function $f: \mathcal{C} (\R^d) \to \R^+ \cup \{\infty \}$.  We say $f$ is monotone if $A \subset B$ implies $f(A) \le f(B)$.
The crudest way to define that a convex set $K$ is large is to say that $K$ is large if it is non-empty, corresponding to Helly's theorem.  
The usual non-quantitative Helly theorems arise when we ask whether $f(K) >0$ or not.  
For example, if we are given a subset $S \subset \R^d$, we could define $f(K) = 0$ if $K \cap S = \emptyset$ and $f(K) = 1$ otherwise.  
This reproduces the $S$-Helly numbers we mentioned earlier. 

Given a discrete set $S \subset \R^d$, one can consider  the function $f(K) = |K \cap S|$, which leads to several new Helly-type results.  Other notions of size are given by the volume of $K$ or its diameter.  We stress that Helly-type theorems based on functions such as $f(K)=\vol (K)$ or $f(K)=\diam (K)$, which vary continuously over the convex sets, yield wildly different results than Helly-type theorems based on functions such as $f(K)=|K \cap S|$, which vary discretely. 

Our first example is a quantitative version of Doignon's theorem.

\begin{theorem}[Aliev, Bassett, De Loera, Louveaux \cite{Aliev:2014va}]  \label{k-doignon}
Let  $d$, $k$  be non-negative integers, there exists a constant $c(k,d)$, depending only on $k$ and $d$, such that if a rational polytope $P_A(b)=\{ x : Ax \leq b\}$ in $\R^d$ has  exactly $k$ integral solutions, then there is a subset of the inequalities of $Ax\leq b$, of cardinality  no more than $c(k,d)$, such that the induced subproblem has exactly the same $k$ integer solutions as $P_A(b)$. Moreover,

\[
c(k,d) \leq \lceil 2(k+1)/3\rceil (2^d-2)+2.
\]
\end{theorem}

Note that the original Doignon-Bell-Scarf theorem is $k=0$.  This theorem implies a version for convex sets.

\begin{corollary}\label{corollary-k-doignon}
 Let  $(X_i)_{i \in \Lambda}$ be a collection of convex sets in $\R^d$,
where at least one of these sets is compact.  If the intersection of every sub-collection of size at most $c(k-1,d)$ contains at least $k$ integer points, then the intersection of the whole family contains at least $k$ integer points.
\end{corollary}

\begin{openproblem}
	Determine the exact value of the Helly number $c(k,d)$.
\end{openproblem}

The value of $c(k,d)$ by Aliev et al. was shown to be optimal for $k=1$, and not to be optimal starting from $k=3$. Recently Gonzalez Merino and Henze showed that the formula above is tight for $k=2$ \cite{gonzalez+henze} and that the Helly number $c(k,d)$ can be used to bound the number of lattice
points on the boundary of a strictly convex set in terms of the number of lattice points in the interior. Of course, one can count points of discrete sets, thus
one might expect similar quantitative Helly numbers for other discrete sets besides lattices.

\begin{definition}
	We say that a discrete set $S \subset \R^d$ has a $k$-quantitative Helly number $h_k(S)$ if it is the smallest positive integer such that the following statement holds.  For every finite family $\F$ of convex sets in $\R^d$, if the intersection of every subfamily of size $h_k(S)$ has at least $k$ points of $S$, then $\cap \F$ has at least $k$ points of $S$.  If no integer satisfies that condition, we say $h_k (S) = \infty$.
\end{definition}

\begin{openproblem}
Characterize the pairs $(S,k)$ for which $h_k(S) < \infty$.
\end{openproblem}

For instance, in \cite{DeLoera:2015wp}, it was shown that the difference of a lattice and the union of some of its sublattices has finite $k$-quantitative Helly numbers for all $k$.  Moreover, the fact that $h_k(S) < \infty$ is enough to give quantitative versions of other classic results in discrete geometry, such as Tverberg's theorem \cite{Tverberg:1966tb}.

Historically, the first quantitative Helly-type theorems were presented by  B\'ar\'any, Katchalski and Pach, and were based on several continuous functions over the convex sets.

\begin{theorem}[Large-volume Helly Theorem; B\'ar\'any, Katchalski, Pach\cite{Barany:1984ed, Barany:1982ga}]
Given a family $\F$ of $n\geq 2d$ convex sets in $\R^d$, if the intersection of every $2d$ members of $\F$ has volume at least one, then the 
intersection of all the members of $\F$ has volume at least $ {d^{-2d^2}}$.
\end{theorem}

The guarantee on $\vol (\cap \F)$ has been improved by Nasz\'odi to $d^{-cd}$ for some absolute constant $c >0$ \cite{Naszodi:2015vi}, which is optimal up to the value of $c$.  Nasz\'odi's methods have been refined in \cite{brazitikos-vol}.  This begs the question of how much the volume guarantee can be improved if the size of the subfamilies one is willing to check for intersection is increased from $2d$ to a higher cardinality.  

\begin{theorem}[De Loera, La Haye, Rolnick, Sober\'on \cite{DeLoera:2015wp}]\label{theorem-precise-volume-helly}
Let $\varepsilon > 0$ and $d$ be a positive integer.  There is an integer $n^{\vol}(d, \varepsilon)$ such that the following statement holds.  
Given a finite family $\F$, if the volume of the intersection of every subfamily of size at most $n^{\vol} (d,\varepsilon)$ is at least one, then
\[
\vol (\cap \F) \ge (1+\varepsilon)^{-1} \sim 1- \varepsilon.
\]
Moreover, the optimal value of $n^{\vol} (d, \varepsilon)$ satisfies $n^{\vol} (d, \varepsilon) = \Theta_d(\varepsilon^{-(d-1)/2})$.
\end{theorem}

The $\Theta_d$ notation hides constants which may depend on the dimension.  We should emphasize the striking difference between Theorem \ref{theorem-precise-volume-helly} and Corollary \ref{corollary-k-doignon}.  If we measure our sets by their volume, there is no way to avoid a loss.  In other words, if we ask for every subfamily of fixed size to have volume at least one, we cannot guarantee 
that the intersection of the whole family has volume at least one.  However, when counting the number of points of a discrete set $S$ in our convex sets, no such loss appears.

Results such as Theorem \ref{theorem-precise-volume-helly} are closely related to finding good approximations of convex sets by polyhedra.  This is a vibrant subfield of convexity; we recommend the surveys \cite{Bronstein:2008dn, gruber1993}.  For the sake of obtaining Helly-type theorems, we require a very specific kind of approximation compatible with the function we wish to quantify.

\begin{definition}\label{definition-hellytype-approx}
Given a function $f: \mathcal{C}(\R^d) \to \R^+\cup \{\infty\}$ we say that $f$ is compatible with polyhedral approximations of convex sets if for every $\varepsilon > 0$ there is a number $n(f,d,\varepsilon)$, depending only on $f$, the dimension $d$, and $\varepsilon$, such that the following holds.  For any convex set $K \subset \rr^d$ with $f(K) < \infty$, there is a polyhedron $P$ such that
\begin{itemize}
	\item $K \subset P \subset \R^d$ ,
	\item $P$ has at most $n(f,d,\varepsilon)$ facets, and
	\item 
	$(1-\varepsilon) f(P) \le f(K)$.
\end{itemize}
Let $n(f,d,\varepsilon)$ be the minimum value which satisfies the conditions above.
\end{definition}

For example, compatible polyhedral approximations exist for $f(K)=\vol(K)$, for $f(K)=\max \{c^T x:  x\in K\}$, and for $f(K)=\int_K g(x) dx$ where $g(x)$ is a non-negative Lipschitz function.  In the last example, if we only ask for $g(x)$ to be continuous with respect to the Lebesgue measure, $f$ may not be compatible with polyhedral approximations of convex sets.

In order to see this, consider for each positive integer $n$ the circle $C_n \subset \R^2$ of radius one around the point $(3n,0)$.  Then, we can consider $\mu_n$ a uniform probability measure on the set $C_n + B_{1/(2n)}(0)$, where $B_{1/(2n)}(0)$ stands for the ball of radius $1/(2n)$ around the origin.  Finally, our function $f$ is simply $f= \sum_{n=1}^{\infty}\mu_n$.  This function is  continuous but not compatible with polyhedral approximations of convex sets.  Indeed, given $\varepsilon>0$, it is clear that if $K=C_n$, for a polygon $P \supset K$, if $f(P)$ is to satisfy the desired conditions, then $P$ would have to be arbitrarily close to $K$ in Hausdorff distance.  This would imply that $P$ has an arbitrarily large number of sides, making hte existence of a universal $n(f,d,\varepsilon)$ impossible.

Using this concept and the properties of such functions $f$, we can prove the following new theorem.  Which is a very general version of Theorem \ref{theorem-precise-volume-helly}.  The optimality of the bound below is discussed after the proof.

\begin{theorem}\label{theorem-general-quantitative-helly}
	Let $d, \varepsilon > 0$ be given.  Let $f: \mathcal{C}(\R^d) \to \R^+\cup \{\infty\}$ be a monotone function which is compatible with polyhedral approximations. Then, for any finite family $\F$ of convex sets in $\rr^d$, if every subfamily $\F'$ of size at most $d \cdot n(f,d,\varepsilon)$ satisfies $f(\cap \F') \ge 1$, then the entire family satisfies $f (\cap \F) \ge 1-\varepsilon$.	
\end{theorem}

The results mentioned in Section \ref{secapps} show that the theorem above can be used to approximate the 
value of the integral of a Lipschitz functions or the volume of the intersection of a family of convex sets.
To prove Theorem \ref{theorem-general-quantitative-helly} we need the following lemma, which, not surprisingly, 
is a direct consequence of Helly's theorem.

\begin{lemma}\label{lemma-hyperplane-containment}
Let $\F$ be a finite family of convex sets in $\R^d$ and $H^+$ be a closed half-space.  If $\cap \F \neq \emptyset$ and $\cap \F \subset H^+$, then there is a subfamily $\G \subset \F$ of size at most $d$ such that $\cap \G \subset H^+$.
\end{lemma}

\begin{proof}
Consider $H^- = \R^d \setminus H^+$. Notice that the family $\F \cup \{	H^-\}$ has empty intersection and is finite.  By Helly's theorem, there must be a subfamily of size at most $d+1$ which has empty intersection.  This subfamily must include $H^-$, since $\cap \F \neq \emptyset$.  The other $d$ elements form the family $\G$ we were looking for.
\end{proof}

\begin{proof}[Proof of Theorem \ref{theorem-general-quantitative-helly}]
	If $f(\cap \F) = \infty$, there is nothing to prove.  Otherwise, assume $f(\cap \F) < \infty$.  We can construct a polyhedron $P$ as in Definition \ref{definition-hellytype-approx} for $K = \cap \F$.  Then, for an arbitrary facet of $P$, consider the half-space determined by that hyperplane which contains $\cap \F$.  Using the lemma above, we can find $d$ sets of $\F$ whose intersection is contained in that half-space.
	
	If we repeat this process for each facet of $P$, we have constructed a family $\F'$ of size at most $d \cdot n(f,d,\varepsilon)$ such that $\cap F' \subset P$.  Then, it suffices to notice that
	\[
	f(\cap \F) \ge (1-\varepsilon) f(P) \ge (1-\varepsilon) f(\cap \F') \ge 1- \varepsilon.
	\]
	\end{proof}

In many cases, one can obtain that the constant $n(f,d,\varepsilon)$ is also a lower bound for the cardinality of the subfamilies to check in order to obtain the conclusion of Theorem \ref{theorem-general-quantitative-helly}.  In other words, in those cases the bound given by Theorem \ref{theorem-general-quantitative-helly} is off by a multiplicative factor of at most $d$.

To do this, it suffices to take a convex set $K$ which does not allow an approximation as in Definition \ref{definition-hellytype-approx}, e.g.,  one that cannot be approximated efficiently with a polyhedron with at most $n(f,d, \varepsilon)-1$ facets, and taking $\F$ to be a sufficiently large family of half-spaces containing $K$.  This idea gives the lower bound of Theorem \ref{theorem-precise-volume-helly}, for instance.

Thus, in order to obtain a quantitative Helly-type theorem for other functions, it suffices to prove the existence of efficient polyhedral approximations (i.e. bound $n(f,d,\varepsilon)$).  The upper bounds for the volume $n(\vol, d, \varepsilon)\sim n^{\vol}(d, \varepsilon)$ can be found in \cite{Anonymous:HS2Q-kvJ}, and the lower bounds in \cite{gruber1993}.  A compactness argument over $\mathcal{C}({\R^d})$ in \cite{DeLoera:2015wp} shows that $n(\diam, d, \varepsilon) < \infty$.  Precise asymptotic bounds $n(\diam, d, \varepsilon) = \Theta_d (\varepsilon^{-(d-1)/2})$ appear in \cite{Soberon-diameter}.

\begin{openproblem}
Determine the largest value $r(d)$ such that the following statement holds.  For every finite family $\F$ convex sets in $\R^d$, if the intersection of every $2d$ of them has diameter at least one, then $\diam (\cap \F) \ge r(d)$.
\end{openproblem}

Is known that $r(d) \ge d^{-2d}$ \cite{Barany:1982ga}, but the conjecture is $r(d) \sim c d^{-1/2}$ for some constant $c$. Brazitikos has recently confirmed the conjecture in the case where the sets are centrally symmetric around the origin \cite{brazitikos-diam}, and provided the following result for the non-symmetric case.

\begin{theorem}[Brazitikos \cite{brazitikos-diam}]
	There is are absolute constants $\alpha >1, c>0$ such that the following is true.  For any finite family $\F$ of convex bodies in $\rr^d$, if the diameter of the intersection of any $\alpha d$ of them has diameter at least one, then the $\diam (\cap \F) \ge c d^{-3/2}$.
\end{theorem}

\subsection{Topological versions} \label{topologicalversions}

The convexity in Helly's theorem can be replaced by topological properties on the sets of
$\F$ and their intersections.  
Indeed, convex sets (and their intersections) are contractible sets, so it is to be expected that such conditions should be sufficient to characterize the intersection of finite families.
The first theorem of this kind was proved by E. Helly himself.

We say a topological space is a \textit{homology cell} if it is non-empty and its homology groups vanish in all dimensions.

\begin{theorem}[Helly \cite{Helly:1930hk}]\label{theorem-topological-Helly-version1}
Given a finite family $\F$ of $n\geq d+1$ open homology cells in $\R^d$, if
\begin{itemize}
	\item every subfamily $\F'$ of $\F$ of at most $d+1$ members has non-empty intersection and
	\item for every proper subfamily $\F'$ of $\F$ with non-empty intersection $\cup \F'$ is a homology cell,
\end{itemize}
then the intersection of all the members of $\F$ is non-empty.
\end{theorem}

The intuitive idea 
explaining this theorem is that if a 
large finite family $\G$ of sets as above has the property that all proper subfamilies are intersecting but $\G$ is not, then $\cup \G$ is 
homologically equivalent to a high-dimensional ball.  This is impossible if $\cup \G$ is embedded in a low-dimensional Euclidean space.

\begin{lemma}\label{lemma-topological-tool}
	If $\G$ is a finite family of open homology cells in a topological space, such that 
	\begin{itemize}
		\item every proper subfamily of $\G$ is intersecting and its union is a homology cell,
		\item $\cap \G = \emptyset$, and
		\item $|\G| = n+1,$
	\end{itemize}
	then $\cup \G$ is homologically equivalent to $S^{n-1}$.
\end{lemma}

\begin{proof}[Proof of Lemma \ref{lemma-topological-tool}]
	We proceed by induction on $n$.  If $n = 1$, the result is true.  Assume that the result is true for $n=k$ and let us consider a family with the properties above and of size $k+2$.
	
	Let $B$ be a set in $\G$, and let $\G':= \G \setminus \{B\}$.  We know $\cap \G' \neq \emptyset$ and $A=\cup \G'$ is a homology cell.  The family $\G'_B = \{K \cap B: K \in \G'\}$ satisfies the condition for $n=k$.  Thus, $A \cap B = \cup \G'_B$ is homologically equivalent to $S^{k-1}$.  By the Mayer-Vietoris sequence we have that $A \cup B = \cup G$ is homologically equivalent to $S^{k}$.  
Intuitively, $A \cap B$ is homologically equivalent to $S^{k-1}$ and adding the cells $A$ and $B$ we are adding the north and south hemispheres of $S^{k}$.
\end{proof}

\begin{proof}[Proof of Theorem \ref{theorem-topological-Helly-version1}]
If the theorem does not hold, let $\G$ a subfamily of minimal cardinality such that $\cap \G = \emptyset$.  Then, by the conditions of the theorem $n = |\G| \ge d+2$ and by the lemma above $\cup \G \sim S^{n-2}$.  Since $\cup \G$ is embedded in $\R^d$ and $n-2 \ge d$, this is not possible.
\end{proof}

Of course, being a homology cell is quite  a strong condition.  The requirements on the sets have been trimmed down.  For instance, there is a classical theorem of Breen that extended Helly's theorem:

\begin{theorem}[Helly's theorem for star-shaped union; Breen \cite{Breen:1990ez}]\label{theorem-breen}
	Let $\F$ be a finite family of compact convex sets in $\R^d$.  Then, every subfamily of $\F$ of size at most $d+1$ has starshaped union if and only if $\cap \F \neq \emptyset$.
\end{theorem}

To produce a topological generalization of Breen's theorem with a proof similar to that of Theorem~\ref{theorem-topological-Helly-version1}, 
we only need to ask for conditions on certain homology groups $H_j(\cdot)$ of the subfamilies. Thus:

\begin{theorem}[Topological Breen theorem; Montejano \cite{Montejano:2014ii}]
	Let $\F$ be a finite family of open subsets of a topological space $X$.  Let $d$ be an integer such that $H_i (U) = 0$ for all $i \ge d$ and every open subset $U$ of $X$.
	
Suppose that $H_{j-2} (\cup \F') = 0$ for any subset $\F' \subset \F$ of size $j$, $1 \le j \le d+1$.  Then $\cap \F \neq \emptyset$.
\end{theorem}

A topological version of Helly's theorem by making requirement on the intersections rather than the unions also holds.

\begin{theorem}[Topological Helly theorem; Montejano \cite{Montejano:2014ii}]\label{theorem-topological-montejano}
	Let $\F$ be a finite family of open subsets of a topological space $X$.  Let $d$ be an integer such that $H_i (U) = 0$ for all $i \ge d$ and every open subset $U$ of $X$.
	
	Suppose that $H_{d-j} (\cap \F') = 0$ for any subset $\F' \subset \F$ of size $j$, $1 \le j \le d+1$.  Then $\cap \F \neq \emptyset$.
\end{theorem}

This uses fewer conditions than an earlier topological Helly-style theorem by Debrunner \cite{Debrunner:1970th}.  Another way to obtain topological versions of Helly's theorem is to work with the \textit{nerve complex} of the family $\F$ instead of $\F$.  Given a family of sets $\F$, the nerve complex of $\F$, denoted by $N(\F)$, is the simplicial complex formed by adding one vertex for each non-empty element in $\F$ and a face for each intersecting subfamily.
It turns out that the nerve complexes of families of convex sets have deep topological properties.  A lucid explanation is contained in a recent survey by Tancer \cite{Tancer:2013iz}.  It should be noted that some topological versions, such as the colorful topological Helly type results of Kalai and Meshulam \cite{Kalai:2005cm}, have application to interesting purely geometric problems.  Two examples are applications to geometric versions of Hall's theorem for graph \cite{HMM} and extensions of the colorful Carath\'eodory theorem \cite{KH15}.

One of the key ideas around Helly theorems, as is noted among some proofs in this survey, is to continuously translate a half-space containing the family until it loses a non-trivial intersection.  The same idea was used by Wegner to show that nerve complexes of families of convex sets in $\R^d$ are \textit{$d$-collapsible} \cite{Wegner:1975eo}.  In other words, the nerve complex of a finite family of convex sets may be brought down to the empty simplex by iteratively removing faces of dimension at most $d-1$ which are contained in exactly one containment-maximal face.  Indeed, the fact that every nerve complex of a finite family of convex sets in $\R^d$ is $d$-collapsible implies Helly's theorem.

Many Helly-type results hold for families whose nerve complexes are $d$-collapsible.  A stronger condition than being $d$-collapsible is being \textit{$d$-Leray}, which means that all homology groups of $N(\F)$ of dimension greater than or equal to $d$ are zero (see \cite{matousektancer} where collapsable and Leray complexes are discussed in detailed). By the nerve lemma (see \cite{Bjorner:1995vi}), or as a consequence of being $d$-collapsible, nerve complexes of families of convex sets in $\R^d$ are $d$-Leray.  In fact, the nerve lemma can be used to provide a very short inductive proof of Helly's theorem.  It comes from the observation that a non-intersecting family of convex sets where every $d+1$ of them do intersect has as a nerve the boundary of a $(d+1)$-dimensional simplex, which is impossible if the family is contained in $\R^d$.

We say that a family $\G$ of subsets of $\R^d$ is a \emph{good cover} if the intersection of every subfamily is either empty or contractible.  Notice that the family of all convex sets in $\R^d$ is  a good cover, and the same application of the nerve lemma shows that the nerve complex of a finite subset of a good cover of $\R^d$ is $d$-Leray.  Kalai and Meshulam showed that good covers are enough to obtain the following result.

\begin{theorem}[Kalai, Meshulam \cite{Kalai:2008kc}]\label{theorem-kalai-covers-disjoint}
	Let $\G$ be a good cover of $\R^d$ and $\F$ be a finite family of sets such that the intersection of any subfamily of $\F$ is the union of at most $k$ elements of $\G$.  Then, if every $k(d+1)$ elements of $\F$ intersect, so does $\F$.
\end{theorem}

This is a generalization of the following result by Amenta.

\begin{theorem}[Amenta, \cite{Amenta:1996eu}]\label{theorem-amenta-disjoint}
Let $\F$ be a family of sets in $\R^d$ such that the intersection of any members of $\F$ can be expressed as the union of at most $k$ pairwise disjoint compact convex sets.  If every $k(d+1)$ or fewer members of $\F$ have a point in common, then there is a point in common to all memebers of $\F$.	
\end{theorem}

The result above gives a clean solution of a conjecture by Gr\"unbaum and Motzkin \cite{Grunbaum:1961fd}.  The conjecture had a more elaborate solution due to Morris \cite{Morristhesis}, whose correctness has been put to question. We give a proof of Theorem \ref{theorem-amenta-disjoint} in Section 3.  If the disjointness assumption is dropped, larger Helly numbers are needed \cite{Alon:1995fs, Matousek:1997di}.

There have been two recent of generalizations of this result.  The first is a topological generalization by Goaoc, Pat\'ak, Pat\'akov\'a, Tancer, and Wagner \cite{Goaoc:2015aaa}, which further generalizes a previous topological extension \cite{ColindeVerdiere:2014gw}.  In this case, asking of requiring the family $\F$ to be a good cover of $\R^d$, only bounds on the first $\lceil d/2\rceil$ Betti numbers over $\Z_2$ of $\cap \mathcal{G}$ are required for each subfamily $\mathcal{G} \subset \F$, which is a significantly weaker condition.  This result also generalizes \cite{Matousek:1997di}. The second generalization is a combinatorial version of Theorem \ref{theorem-amenta-disjoint}  by Eckhoff and Nischke.  The arguments of their proof rely on fixing and extending some of the arguments by 
Morris \cite{Morristhesis}.  To describe this last result we need the following definitions.

Given a family $\F$ of sets, we say $\F$ is
\begin{itemize}
	\item \textit{intersectional} if for any finite subfamily $\G$, $\cap \G$ is either in $\F$ or empty, and
	\item \textit{non-additive} if for any finite subfamily $\G$ of pairwise disjoint subsets of $\F$, if $|\G| \ge 2$ then $\cup \G \not\in \F$.
\end{itemize} 

Notice that the family of convex sets in $\R^d$ is intersectional and non-additive.

\begin{theorem}[Eckhoff, Nischke \cite{Eckhoff:2009kv}]\label{theorem-eckhoff-nischke}
	Let $\mathcal{C}$ be an intersectional non-additive family.  If $\F$ is a finite family of sets such that the intersection of any $k$ or fewer members of $\F$ is a disjoint union of at most $k$ members of $\mathcal{C}$, then $h(\F) \le k h(\mathcal{C})$.
\end{theorem}

Given a good cover $\G$ of $\R^d$, the family of finite intersections of $\G$ is an intersectional non-additive family.  Thus, Theorem \ref{theorem-kalai-covers-disjoint} is also implied by Theorem \ref{theorem-eckhoff-nischke}, which is proved by completely different methods.  The results of Eckhoff and Nischke and those of Goaoc et al. do not seem to imply each other.

Topological Helly-numbers are a form of moving to non-convex spaces. Similarly, there are some Helly-type results for curves and
for algebraic varieties. Consider for example the following theorem about convex curves.

\begin{theorem}[Helly for boundary of convex planar curves, Swanepoel \cite{Swanepoel:2003bg}] \label{swanepoel}
	Given a convex body $K \subset \R^2$, and $\F$ a finite family of homothets (including translates) of the boundary $\partial K$, if every four elements in $\F$ have a point in common, all of them do.
\end{theorem}

One can successfully look at non-linear algebraic versions of Helly's theorem, where the sets in the family $\F$ are algebraic 
hypersurfaces defined as the zeros of a polynomial.  There is indeed a  Helly-type theorem for hypersurfaces defined by 
homogeneous polynomials. This was proved by Motzkin \cite{motzkin} and later reproved by  Deza and Frankl \cite{deza+frankl,frankl}. 


\begin{theorem}[Motzkin and Deza, Frankl \cite{motzkin, deza+frankl,frankl}]\label{thm:DF-Helly} 
	Let $f_1,\dots,f_m$ be a system of homogeneous polynomials in $n$ variables and coefficients over a field. Define $d=\max\{\operatorname{deg}(f_i)\}$. 
	Suppose that every subset of $p={{n+d}\choose{d}}$ polynomials $\{f_{i_1},\dots,f_{i_p}\}\subset \{f_1,\dots,f_m\}$ has a solution. Then the entire system of polynomial equations $\{f_1,\dots,f_m\}$ does as well. 	
\end{theorem}

As an example that resembles Theorem \ref{swanepoel},  take $\F=\{f_1(x,y),\dots,f_s(x,y)\}$ a family of affine real plane curves of degree at 
most $d$. If every $\delta={d+2 \choose 2}$ of the curves have a common real intersection point, then all curves in $\F$ have a real intersection point. 
The proof of Theorem \ref{thm:DF-Helly} uses the linear algebraic structure of the vector spaces of polynomials of degree $d$ in $n$ variables. This result was later
extended to \emph{quasialgebraic sets} in \cite{quasialgebraic}.

Other ``algebraic versions'' of Helly's theorem have been considered in recent years. E.g. Tropical geometry is a very active area where the usual arithmetic operations are replaced by the max-plus counterparts, i.e., maximum and addition are the two binary operations in $\R$ . In \cite{Gaubert:2010ke} the authors presented several exciting tropical versions of the classical theorems of convexity (see \cite{Briec:2004eb} for the first tropical version of Helly's theorem and the survey \cite{trop2}).
Finally, \cite{basugabi} shows a Helly-type theorem for finite families of subsets of $\Bbb{R}^n$ which are graphs of monotone maps definable in some $o$-minimal structure over $\Bbb{R}$.
This is a rather powerful framework, since one of the most famous examples of  $o$-minimal structures is that of  semi-algebraic sets. On the other hand, graphs of monotone maps are, generally, non-convex, and their intersections, unlike intersections of convex sets, can be topologically complicated. 

\subsection{Fractional versions and versions with tolerance}

The hypothesis statement of many Helly-type theorems requires that all subfamilies of cardinality $h$ satisfy a certain property. In this 
Subsection we discuss two versions which allow some relaxation of this assumption. 

The fractional version of Helly's theorem is obtained by changing the condition that \textit{all} $(d+1)$-tuples are intersecting to the condition that \textit{most} $(d+1)$-tuples are intersecting.

\begin{theorem}[Fractional Helly theorem; Katchalski, Liu \cite{Katchalski:1979bq}] 
Given a family $\F$ of $n$ convex sets in $\R^d$ and a parameter 
$\alpha >0$,   if at least $\alpha {n \choose d+1}$ of the $(d+1)$-tuples of $\F$ 
 have a common intersection, then there exists a point contained in at least $\beta n$ objects
of $\F$. Here $\beta$ is a positive constant depending only on $\alpha$ and $d$.
\end{theorem}

\begin{proof}
	We may assume without loss of generality that the sets in $\F$ are compact.  Let $v$ be a direction.  We consider a $v$-halfspace to be a set of the form $\{x \in \R^d: \langle x , v \rangle \ge \alpha \}$ for some $\alpha$.  Given an intersecting $d$-tuple $A$ of $\F$, we can consider $H_A$ to be the containment-minimal $v$-halfspace such that $\left( \cap A \right) \bigcap H_A$ is not empty, which exists by the compactness of the sets.  Moreover, $v$ can be chosen so that for each $A$ as above $\left( \cap A \right) \bigcap H_A$ consists of a single point $p_A$.
	
	Given an intersecting $(d+1)$-tuple $B\subset \F$, we can assign to it its $d$-tuple $A$ for which $H_A$ (its associated containment minimal $v$-halfspace) is containment-maximal. Helly's theorem immediately gives that $B \cup \{H_A\}$ is intersecting, so all the sets in $B$ contain $p_A$.  A simple counting argument shows that if a positive fraction $\alpha$ of the $(d+1)$-tuples is intersecting, then this process assigns the same $d$-tuple $A$ at least $\beta n$ times for some $\beta$ depending only on $d$ and $\alpha$.  Thus, at least $\beta n$ sets contain $p_A$.
\end{proof}

The proof above only gives $\beta \ge \alpha /(d+1)$, but it is quite versatile. Several of the results below are a consequence of this method with little modifications needed.  It was later shown by Kalai \cite{Kalai:1984bg} and Eckhoff \cite{Eckhoff:1985hr} that $\beta = 1 - (1-\alpha)^{1/(d+1)}$, which in particular gives $\beta \to 1$ as $\alpha \to 1$ (Katchalski and Liu showed this limit as well, but not the optimal value for $\beta$).  Kalai's result has been extended to require weaker topological conditions on the sets \cite{hell-thesis, colin2012multinerves} giving topological versions of the fractional Helly theorem and the $(p,q)$ theorem.  At first glance the fractional Helly theorem shows the robustness of Helly's theorem.  In reality it goes further than that.
For instance, an integer version of the result above is also possible, as shown below.

\begin{theorem}[Alon, Kalai, Matou\v{s}ek, Meshulam \cite{Alon:2002wz}]
	Given a family $\F$ of $n$ convex sets in $\R^d$ and a parameter 
$\alpha >0$  such that $\alpha {n \choose {2^d}}$ of the  $2^d$-tuples of  $\F$ 
 have an integer point in their intersection, there exists an integer point contained in at least $\beta n$ sets of $\F$. Here $\beta$ is a positive constant depending only on $\alpha$ and $d$
\end{theorem}

Again, $\beta \to 1$ as $\alpha \to 1$, showing the robustness of Doignon's theorem.

The technique can be pushed to show the robustness of quantitative Helly theorems for discrete sets $S$.

\begin{theorem}[Rolnick, Sober\'on \cite{Soberon:2015ts}]\label{theorem-fractional-quantitative}
Let $S \subset \R^d$ be a discrete set and $k >0$ an integer such that $h_k(S) < \infty$.  Given a parameter $\alpha >0$, let $\F$ be a family of $n$ convex sets in $\R^d$ such that the intersection of at least $\alpha {n \choose {h_k(S)}}$ of the  $h_k(S)$-tuples of  $\F$ 
 have at least $k$ points of $S$ in their intersection.  Then, there exists a subset of $k$ points of $S$ contained in at least $\beta n$ sets
of $F$. Here $\beta$ is a positive constant depending only on $\alpha, S$ and $d$. 
\end{theorem}

\begin{theorem}[Rolnick, Sober\'on \cite{Soberon:2015ts}]\label{theorem-fractional-volumetric}
Let $\varepsilon >0$ and $d>0$.  There is a constant $n_2(d, \varepsilon)$ such that the following holds.  Given a parameter $\alpha >0$  and a family $\F$ of $n$ convex sets in $\R^d$ such that the intersection of at least $\alpha {n \choose {n_2(d,\varepsilon)}}$ of the  $n_2(d, \varepsilon)$-tuples of  $\F$ 
 have volume at least one, then there exists a set $C\subset \R^d$ of volume at least $1-\varepsilon$ contained in at least $\beta n$ sets
of $F$. Here $\beta$ is a positive constant depending only on $\alpha, \varepsilon$ and $d$.

Moreover, $n_2(d, \varepsilon) = O_d\left(\varepsilon^{-(d^2-1)/4}\right)$. 
\end{theorem}

The dependence of $n_2(d,\varepsilon)$ on $\varepsilon$ in Theorem \ref{theorem-fractional-volumetric} is not clear.  It is possible for this fractional Helly to have a much nicer behavior than its Helly counterpart.  For certain functions, such as the diameter, this behavior can be observed \cite{Soberon-diameter}.  Thus, it would be interesting to know the answer to the following problem.

\begin{openproblem}
In Theorem \ref{theorem-fractional-volumetric}, is it possible to conclude that the resulting set $C$ has volume at least $1$ (i.e. can $\varepsilon$ be dropped entirely from the statement)?	
\end{openproblem}

The first indication that the fractional Helly theorem carries essentially different information than Helly's theorem is the following surprising result of Matou\v{s}ek and B\'ar\'any.  It says that for the integer fractional Helly theorem it is not necessary to check the $2^d$-tuples in order to get a positive-fraction intersection result, but the $(d+1)$-tuples are sufficient.  Namely

\begin{theorem}[B\'ar\'any, Matou\v{s}ek \cite{Barany:2003wg}]
	Given a family $\F$ of $n$ convex sets in $\R^d$ and a parameter 
$\alpha >0$  such that $\alpha {n \choose {d+1}}$ of the  $(d+1)$-tuples of  $\F$ 
 have an integer point in their intersection, then there exists an integer point contained in at least $\beta n$ sets
of $\F$. Here $\beta$ is a positive constant depending only on $\alpha$ and $d$.
\end{theorem}

In this case, we do not get $\beta \to 1$ as $\alpha \to 1$, since $h(\Z_d) = 2^d$.  The proof relies on an argument reducing properties of $\Z^d$-convexity to $\R^d$-convexity \cite[Cor. 2.2]{Barany:2003wg}, which allows the use of Helly-type theorems for which the Helly number is $d+1$ (in particular, the colorful Helly theorem).  The question of which sets $S \subset \R^d$ allow for a fractional Helly theorem that only needs to check $(d+1)$-tuples instead of $h(S)$-tuples was answered by Averkov and Weismantel.

\begin{theorem}[Averkov, Weismantel \cite{Averkov:2010tva}]
	Let $S \subset \R^d$ be a closed set with $h(S) < \infty$.  Given a family $\F$ of $n$ convex sets in $\R^d$ and a parameter 
$\alpha >0$  such that $\alpha {n \choose {d+1}}$ of the  $(d+1)$-tuples of  $\F$ 
 have a point of $S$ in their intersection, there exists a point of $S$ contained in at least $\beta n$ sets
of $\F$. Here $\beta$ is a positive constant depending only on $\alpha$, $S$ and $d$.
\end{theorem}

It is not known if the same improvements hold for the volumetric and quantitative versions of the fractional Helly theorem.  Moreover, since the result of 
B\'ar\'any and Matou\v{s}ek shows that for fractional versions of Helly's theorem one may drastically weaken the conditions, it is natural to ask the following question.

\begin{openproblem}
Determine if the size of the subfamilies one needs to check in Theorem \ref{theorem-fractional-quantitative} can be lowered from $h_k (S)$.  
In particular, is it possible to obtain a bound that depends on $k$ and $d$ but not on $S$?
\end{openproblem}

 The results of Weismantel and Averkov show that this is the case when $k=1$.
It should be mentioned that additional conditions on the sets allow for improved bounds on $\beta$, or relaxed requirements on the size of the subfamilies needed to check. 
For example the authors of \cite{Barany:2014vt}  investigated fractional Helly theorems for finite families of boxes. Their main result is the fractional version of Theorem \ref{theorem-boxes-helly-simple} :

\begin{theorem}[B\'ar\'any, Fodor, Mart\'inez-P\'erez, Montejano, Oliveros, P\'or \cite{Barany:2014vt}]
	Let $\F$ be a family of $n$ boxes in $\R^d$, and let $\alpha \in \left(1 - \frac{1}{d}, 1\right]$ be a real number. There exists a real number $\beta (\alpha) > 0$ such that if there are $\alpha{n \choose {2}}$ intersecting pairs in $\F$, then $\F$ contains an intersecting subfamily of size $\beta n$.
\end{theorem} 

  An example shows that the bounds for $\alpha$ are best possible.

\begin{openproblem}
Is there a fractional version of Theorem \ref{theorem-karasev-translates}?	
\end{openproblem}
  
 Another variation of the fractional Helly theorem is a fractional Breen theorem \cite{Barany:2006jh}. There, when many $(d+1)$-tuples of a family of convex sets have a star-shaped union, then the conclusion is that many of the sets must have a point in common.   
   
\begin{openproblem}
Can one prove similar specialized fractional Helly theorems for restricted families of convex bodies besides boxes? 
\end{openproblem}
   
Fractional Helly theorems work in ambient spaces much different than $\rr^d$.  Matou\v{s}ek showed a purely combinatorial version of the fractional Helly theorem, which works on any set system $\G$ of bounded VC-dimension (we define this below, but recommend \cite{Matousek:2009ur} for details).  It should be stressed that Matou\v{s}ek's result gives a large family of fractional Helly theorems for families of sets that do not admit a Helly theorem.  Earlier examples of fractional Helly theorems that did not have a Helly theorem were already known (see Theorem \ref{theorem-fractional-hyperplanes}, for instance). 

Given a set $X$ and a set system $\G$ on $X$, we say that the VC-dimension of $\G$ is the largest size of a set $A$ that is \textit{shattered} by $\G$.  A set $A$ is shattered if every subset of $A$ can be written as $A \cap G$ for some $G \in \G$.  For example, the family of hyperplanes in $\R^d$ has VC-dimension $d$.  Plenty of examples of geometric set systems with bounded VC-dimension can be found in \cite{Matousek:2009ur}.

The fractional Helly number given by Matou\v{s}ek is in terms of the \textit{dual shatter function} of the set system.  The dual shatter function of a set system $\G$ is a function $\pi^*_{\G}: \N \to \N$ where $\pi^*_{\G} (m)$ is the maximum number of non-empty fields of the Venn diagram of $m$ subsets of $\G$.  In other words, given a set $A$ of $m$ subsets of $\G$, it is the largest possible number of subsets $B \subset A$ such that there are elements contained in every set of $B$ but in no set of $A \setminus B$.  It is known that $\pi^*_{\G} (m) = o(m^k)$ for some $k$ if and only if the VC-dimension of $\G$ is bounded.

\begin{theorem}[Matou\v{s}ek \cite{Matousek:2004cs}]\label{theorem-fractional-vc-dimension}
	Let $\F$ be a set system whose dual shatter function satisfies $\pi^*_{\F}$ satisfies $\pi^*_{\F}(m) = o(m^k)$.  Then, for any finite subfamily $\F' \subset \F$ where $\alpha {{|\F'|}\choose{k}}$ of the the $k$-tuples of $\F'$ are intersecting, there is an intersecting subfamily of $\F'$ of size at least $\beta |\F'|$, where $\beta>0$ only depends on $\alpha$ and $k$. 
\end{theorem}

An important corollary of this result is a fractional Helly theorem for semi-algebraic sets.  Recall that a semi-algebraic set in $\R^d$ is defined by a Boolean combination of a bounded number of polynomial inequalities of bounded degree.  Here the maximum of the number of polynomials used and their degrees is called the \textit{complexity} of the system.

\begin{corollary}\label{corsemialgfrac}
	For every fixed $B$, and $\F$ a finite family of \textit{semi-algebraic} sets in $\R^d$ of complexity at most $B$, if at least $\alpha {{|\F|}\choose{d+1}}$ of the $(d+1)$-tuples of $\F$ intersect, there is a subfamily of $\F$ of size at least $\beta |\F|$ which is intersecting.  The constant $\beta>0$ only depends on $B, d$, and $\alpha$.
\end{corollary}

\begin{openproblem}
Is there an integral version of Corollary \ref{corsemialgfrac}?	
\end{openproblem}

The fractional Helly theorem for semi-algebraic sets is not a surprise. The intersection patterns of these families of sets are extremely rich in combinatorial properties (see, for instance, \cite{Conlon:2014fx, Fox:2015uo, Alon:2005hj} and the references therein).

A stronger notion of robustness of Helly's theorem than the one presented by the fractional versions is found in the versions with tolerance.  In this case, instead of allowing a positive-fraction subfamily not to be included in the intersecting subfamily, we only allow avoidance of a subfamily of fixed size.

These results are related to the Erd\H{o}s-Gallai numbers $\eta(d+1,t+1)$, which are constants that come up when studying vertex covers in hypergraphs (see  \cite{AG61} for a proper definition).  The following bound was obtained by Tuza \cite{T89}.
\[
\eta(d+1,t+1) < \binom{d+t+1}{d} + \binom{d+t}{d}.
\]
In particular, the inequality above implies that the case $t=0$ the next theorem directly implies Helly's theorem.
\begin{theorem}[Helly with tolerance; Montejano, Oliveros \cite{Montejano:2011vs}] Let $\eta(d+1,t+1)$ be the Erd\H{o}s-Gallai number for $d+1$ and $t+1$.
	Given a finite family $\F$ of convex sets in $\R^d$, if out of every $\eta(d+1,t+1)$ of the members of $\F$ there is a point in the intersection of all, with the exception of at most $t$ sets, then there is a point in the intersection of all members of $\F$, with the exception of at most $t$ sets. 
\end{theorem}

The proof of this theorem is a consequence of combinatorial properties of intersection hypergraphs. Their result is stated in a very general form \cite[Thm. 1.1]{Montejano:2011vs}.  If we apply this to $S$-Helly numbers, we immediately obtain the following theorem.

\begin{theorem}[Helly with tolerance over $S \subset \R^d$] Let $k>0$ be a integer and $S \subset \R^d$ be a subset with a finite quantitative Helly number $h_k(S)$.  Given a finite family $\F$ of convex sets in $\R^d$, if out of every $\eta(h_k(S),t+1)$ sets in $\F$ there are $k$ points of $S$ contained in all, with the exception of at most $t$ sets, then there is a set of $k$ points of $S$ contained in all members of $\F$, with the exception of at most $t$.
\end{theorem}

\subsection{Piercing numbers and $(p,q)$ theorems}

When Helly's condition does not hold there may not be a single point in the intersection of all members of a family of convex sets.  However, it is still possible that using few points we can intersect all the sets in the family, effectively bounding its \textit{piercing number}. More precisely, given a family $\F$ of non-empty convex sets in $\R^d$, we say that the piercing number of $\F$, denoted by $\pi(\F)$ is the minimum number of points needed to intersect every member of $\F$.  If no finite set can accomplish this, we say $\pi (\F) = \infty$. 

A popular weakening of Helly's condition is the so called $(p,q)$ property.  We say that a family $\F$ of convex sets in $\R^d$ has the $(p,q)$ property if \textit{out of every $p$ sets in $\F$, there are $q$ which are intersecting}.  A classic conjecture of Hadwiger and Debrunner \cite{Hadwiger:1957we} asked if this property was enough to bound the piercing number of the family. This was answered positively by Alon and Kleitman.

\begin{theorem}[$(p,q)$ theorem; Alon, Kleitman \cite{Alon:1992gb}]
	Given $p \ge q \ge d+1$, there is a constant $c = c(p,q,d)$ such that any finite family $\F$ of convex sets in $\R^d$ with the $(p,q)$ property satisfies $\pi (\F) \le c$.
\end{theorem}

We denote by $c(p,q,d)$ the optimal constant that fits the theorem above.  A simplified proof can be found in \cite{Alon:1996uf, Matousek:2002td}.  A survey of Eckhoff collected many results around the $(p,q)$ theorem \cite{Eckhoff:2003ed}.  Even though there are plenty of existence results for $(p,q)$ type problems, finding efficient bounds for $c(p,q,d)$ is extremely difficult.  As noted in \cite{Eckhoff:2003ed}, by following the original proof of the $(p,q)$ theorem one obtains $c(4,3,2) \le 276$, which is the first non-trivial case.  The conjecture is $c(4,3,2) = 3$.  For this problem, the best bounds are $3 \le c(4,3,2) \le 13$ \cite{Kleitman:2001vo}.

\begin{openproblem}
	Improve the bounds on $c(4,3,2)$.
\end{openproblem}

A key step in the proof the $(p,q)$ theorem is the use of the fractional Helly theorem.  The other main ingredient one needs to check is the existence of weak $\varepsilon$-nets for the families of sets in question (for convex sets, see \cite{Alon:2008ek}).  Once the existence for the appropriate type of $\varepsilon$-nets has been cleared we obtain variations of the $(p,q)$ theorem.  Thus, many variations of the fractional Helly theorem imply $(p,q)$ type results.

 This yields hyperplane transversal (as  Theorem \ref{theorem-fractional-hyperplanes}), quantitative (as theorems \ref{theorem-fractional-quantitative} and \ref{theorem-fractional-volumetric}) and bounded VC-dimension (as Theorem \ref{theorem-fractional-vc-dimension}) $(p,q)$ theorems as a consequence, each contained in the corresponding reference.  In all these variations, the key modification is either the set system in question or the conditions on the intersection.

 However, variations of the $(p,q)$ theorem where the local intersection condition is modified are also possible.  For instance, a colorful version of the fractional Helly theorem almost immediately gives the following result.

\begin{theorem}[colorful $(p,q)$ theorem; B\'ar\'any, Fodor, Montejano, Oliveros, P\'or \cite{Barany:2014bp}]
Let $p,q,d$ be positive integers with $p \ge q \ge d+1$.  Then there is an integer $M(p,q,d)$ such that the following holds.  Given finite families $\F_1, \F_2, \ldots, \F_p$ of convex sets such that for every choice $F_1 \in \F_1, \ldots, F_p \in \F_p$ there are $q$ sets of $\{F_1, \ldots, F_p\}$ intersecting, there are $q-d$ indices $i \in \{1,\ldots, p\}$ for which $\pi (\F_i) \le M(p,q,d)$.	
\end{theorem}

If one wishes for improved bounds on the piercing number of the family, strengthening the $(p,q)$ condition helps.  For example, we can say that a family $\F$ of convex sets has the $(p,q)_r$ property if \textit{out of every $p$ convex sets of $\F$, there are $r$ of the $q$-tuples which are intersecting}.

\begin{theorem}[Montejano, Sober\'on \cite{Montejano:2011cz}]
	Let $\F$ be a finite family of convex sets in $\R^d$ and $p \ge q \ge d+1$ be positive integers.  If $\F$ has the $(p,q)_r$ property, $|\F| \ge p$ and 
	\[
r > {\binom{p}{q}}-{\binom{p+1-d}{q+1-d}},
\]
then $\pi (\F) \le p-q+1$.
\end{theorem}

Let us sketch a proof of the result above, as it does not rely on a fractional Helly theorem, but rather uses some of the topological tools described earlier to reduce the problem to dimension one.

\begin{proof}
The reader may notice that the condition implies that every $(d-1)$-tuple of $\F$ is intersecting.  Let $v$ be a direction such that the intersection of every $d$-tuple is either empty or has a $v$-directional minimum at a single point.  Let $A$ be the intersecting $d$-tuple whose $v$-directional minimum $p$ is maximal.  As before, every set $K \in \F$ such that $\{K\} \cup A$ is intersecting contains $p$.

Let $\F_1$ be the subfamily of $\F$ of sets which contain $p$, and $\F_2 = \F \setminus \F_1$.  Consider the half-space $H = \{x: \langle x, v\rangle < \langle p, v\rangle\}$, and define $\F' = \{K \cap H : K \in \F_2\}$.  Similarly, let $A' = \{K \cap H : K \in A\}$.  It suffices to show that $\pi (\F') \le p-q$.  If $|\F'\cup A'| \ge p$, it still satisfies the $(p,q)_r$ condition.  Otherwise, showing that $\pi (\F) \le p-q+1$ can be done by extending $\F_2 \cup A$ to a $p$-tuple and using a brute-force argument.

The family $A'$ is not intersecting and of size $d$, so by \cite{Bracho:2002wj} there is a hyperplane $T$ with orthogonal projection $p: \R^d \to T$ such that $p(A)$ is still not intersecting.  By Lemma \ref{lemma-topological-tool}, $\cup (p(A))$ is homologically equivalent to $S^{d-2}$, which implies that it splits $T$ into two connected components, one of which is bounded.  Let $x$ be a point in the bounded component and consider the line $\ell = p^{-1} (x)$.  It suffices to show that $\F_{\ell}= \{K \cap \ell: K \in \F'\}$ has the $(p-d,q-d+1)$-property, as then we would be able to pierce it with $p-q$ points, since $c(p',q',1) = p'-q'+1$ for $p' \ge q' \ge 2$.

By the construction of $\ell$, every convex set $K$ that intersect every $(d-1)$-tuple of $A$ has a point of $\ell$. However, a purely combinatorial argument shows that for any $(p-d)$-tuple $B$ in $\F'$, since $B \cup A'$ has the $(p,q)_r$ property with $r$ as in the theorem, there must be a $(q-d+1)$-tuple $C$ of $B$ such that the union of $C$ and any $(d-1)$-tuple of $A'$ is intersecting, as desired.
\end{proof}

\subsection{Transversal theorems}

Many Helly-type theorems arise when one changes the conditions that the sets intersect to that of having a transversal by an affine subspace. 
These tantalizing variations often required additional conditions on the family.  Notice that we can think of transversality theorems as intersection theorems.
One way is to observe that if a set of convex sets in $\R^d$ is pierced through by a linear subspace $L$ of dimension $k$, then it means that, under the orthogonal projection to
$L$, all the projected convex sets in $\R^{d-k}$ must intersect.  Another way is to replace each convex set by the space of potential transversals to it, as a subset of a Grassmanian manifold.  Thus, a Helly-type transversality theorem for a family of convex set is equivalent to a Helly-type theorem for the intersection of this new family. In the following theorems about hyperplane transversals, the assumption of convexity on $\F$ can be replaced by connectedness without a problem.

 It is clear that a piercing or transversal line crosses the sets in the family in some order when read from left to right.
When the sets are disjoint one can assume the sets are ordered by such a line. It is intuitively clear such linear ordering of sets is a necessary, although not sufficient, condition for the existence of a transversal. Now we are ready to state the first ever transversal theorem, given by Hadwiger in 1957.

\begin{theorem}[Hadwiger \cite{Hadwiger:1957tx}]
	Let $\F$ be a finite family of pairwise disjoint connected sets in the plane, and $\lambda$ be a linear ordering of the sets.  
	If every three elements of $\F$ have a line transversal that intersects them consistently with $\lambda$, then there is a line transversal to the whole family.
\end{theorem}

Precise conditions for the existence of hyperplane transversals were determined by Pollack and Wenger.  Given a family $\F$ of connected sets in $\R^d$, we say there is a \textit{consistent $k$-ordering of $\F$} if there is a function $f: \F \to \R^k$ such that for any two subfamilies $A, B \subset \F$ with $|A|+|B|=k+2$ and $A \cap B = \emptyset$, if 
$\conv\{f(a) : a \in A\} \cap \conv\{f(b) : b \in B\} \neq \emptyset$, then $\conv (\cup A) \cap \conv (\cup B) \neq \emptyset$.  In other words, the family has all the Radon partitions of a family of points in $\R^k$.

\begin{theorem}[Pollack, Wenger \cite{Pollack:1990cna}]
	A finite family of connected sets in $\R^d$ has a hyperplane transversal if and only if there is a consistent $k$-ordering of $\F$ for some $0 \le k \le d-1$.
\end{theorem}

The case $k=1, d=2$ implies Hadwiger's theorem, and the pairwise disjointness is not needed.  Since a consistent $k$-ordering only needs to be checked for the $(k+2)$-tuples, this constitutes a Helly-type theorem.  It should be emphasized that although the condition of $k$-ordering is necessary for a Helly-type theorem, it is not for the fractional version.

\begin{theorem}[Alon, Kalai \cite{Alon:1995fs}]\label{theorem-fractional-hyperplanes}
	Given a family $\F$ of $n$ connected sets in $\R^d$ and a parameter 
$\alpha >0$  such that $\alpha {n \choose {d+1}}$ of the  $(d+1)$-tuples of  $\F$ 
 have a hyperplane transversal, then there is a hyperplane transversal to at least $\beta n$ sets
of $\F$. Here $\beta$ is a positive constant depending only on $\alpha$ and $d$.
\end{theorem}

Pollack and Wenger's result can be colored.  If we have a colored family $\F$ of connected sets in $\R^d$, we say there is a \textit{consistent colored $k$-ordering of $\F$} if there is a function $f: \F \to \R^k$ such that for any two subfamilies $A, B \subset \F$ with $|A|+|B|=k+2$, $A \cap B = \emptyset$ and $A \cup B$ rainbow (i.e. not repeating any color), if $\conv\{f(a) : a \in A\} \cap \conv\{f(b) : b \in B\} \neq \emptyset$, then $\conv (\cup A) \cap \conv (\cup B) \neq \emptyset$.  In other words, the family has all the colored Radon partitions of a colored family of points in $\R^k$.

\begin{theorem}[Holmsen, Rold\'an-Pensado \cite{Holmsen:2015eh}]
	For $d \ge 1$ and $d-2 \le k \le d-1$, there is number $r(d,k)$ such that the following statement holds.  
	Given a family $\F$ of connected sets in $\R^d$, colored with $r$ colors, if there is a consistent colored 
	$k$-ordering of $\F$, then there is a color that admits a hyperplane transversal to all its sets.
\end{theorem}

This extends a result of Arocha, Bracho and Montejano \cite{Arocha:2008ti}, that showed $r(2,1)=3$ (i.e. a colorful version of Hadwiger's theorem).
Holmsen and Rold\'an-Pensado obtained the general bound $r(d,d-1) \le 2d^2+3$ and a new proof of $r(2,1)=3$.  However, the optimal number of colors needed is still unknown.

\begin{openproblem}[\cite{Holmsen:2015eh}]
Determine if $r(d,d-1) = d+1$ for all $d$.	
\end{openproblem}

The problem above has a positive answer if the family has some additional separation properties on the sets \cite{Oliveros:2008uk}.  Other colorful versions of transversal results are contained in \cite{Mon13transv}.
If one wishes to work without the $k$-ordering condition, it is possible to obtain results by relaxing the conclusion of the theorem.

Given a convex body $K$ in the plane, $\lambda >0$ and family $\F =\{x_1 + K, \ldots, x_n + K\}$ of translates of $K$, we consider $\lambda \F =\{ x_1 + \lambda K, \ldots, x_n + \lambda K\}$.  A classic question of Gr\"unbaum \cite{Grunbaum:1958uu} is to find the minimal $\lambda = \lambda (K,s)$ such that if every $s$ elements of $\F$ have a line transversal, $\lambda \F$ has a line transversal.

For instance if $B$ is a disk, it is known that $\lambda(B,3) \ge \frac{1+\sqrt{5}}{2}$ (see Figure \ref{fig:t3}) but the exact value is still open.

\begin{openproblem}\label{openproblem-t3}
Prove that $\lambda(B,3)=\frac{1+\sqrt{5}}{2}$.	
\end{openproblem}

\begin{figure}
\includegraphics[scale=0.4]{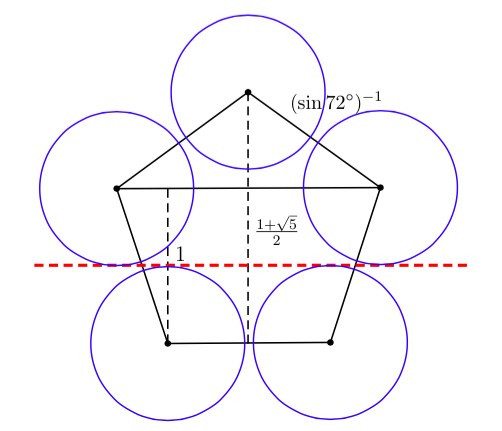}
\caption{Construction showing that the constant in Problem \ref{openproblem-t3} cannot be improved.}
\label{fig:t3}
\end{figure}

It is known that $\lambda(B,4) = \frac{1+\sqrt{5}}{2}$ \cite{JeronimoCastro:2007fk} and $\lambda(B,3) < 1.79005$ \cite{JeronimoCastro:2011cq}.  Additionally, if we ask the balls to be pairwise disjoint, then there is a Helly-type theorem for line transversals, with the Helly number being five \cite{tverberg1989proof}.  This result extends to higher dimensions.

\begin{theorem}[Cheong, Goaoc, Holmsen, Petitjean \cite{Cheong:2009hc}] \label{cheongetal}
Let $\F$ be a family of at least $4d-1$ pairwise disjoint unit balls in $\R^d$.  If every $4d-1$ have a line transversal, then $\F$ has a line transversal.
\end{theorem}

\section{Computational Applications of Helly-type theorems} \label{secapps}

Helly's theorem has many applications, sometimes in unexpected situations, from voting theory \cite{apps-voting} to problems
for clustering incomplete data points \cite{apps-clusteringincomplete}, covering problems \cite{apps-coveringball}, and recently in 
a new algorithmic approach to the interpolation of data by differentiable functions \cite{apps-interpolation}. Once more, for sake of space
and coherence of narrative, we mostly discuss optimization algorithms where Helly numbers and Helly-type theorems play a central role.

Helly's theorem plays a basic role in the theory of convex optimization algorithms.  
Given $n$ convex constraints (their feasible solutions being convex sets) it tells us that 
checking all subfamilies of size $d+1$ for empty intersection (a common solution) would either produce 
a subset of $d+1$ constraints that certifies that the entire set has empty intersection (no common solution), or
that some common solution exists.  Checking for the intersection (or lack thereof) among every subset of $d+1$ constraints is not a
very efficient algorithm to solve convex programs; surely we can do much better. 
But  the basic idea of checking only small subfamilies of constraints to determine a property 
of the entire family set system shows the way towards developing randomized
optimization algorithms that leverage other Helly-type  theorems. 

\subsection{Frameworks and Algorithms}

Linear optimization is perhaps the central problem in the theory of optimization, and it is often
used as a subroutine to solve much more complicated problems which include non-linear or integrality constraints.
For an introduction to linear programming see \cite{schrijverlpbook}. A {\it Linear Program} consists of 
a pair $(\HH,\omega)$, where $\HH$  is a family of  $n$ linear halfspace
constraints in $\real^d$, and $\omega$ is a linear function on $\real^d$. 
The output of a linear programming algorithm should be either a point in the intersection $\cap \HH$ minimizing $\omega$, 
if the intersection is non-empty, or a set of $d+1$ constraints certifying that the intersection 
is empty. 

Helly's theorem describes the combinatorics of linear halfspace intersections, 
and it plays a natural role in 
combinatorial random sampling algorithms that were developed mainly to solve the linear programming 
problem.  
Recall, a combinatorial algorithm accesses the input only through primitive operations
which produce discrete results such as Boolean value or a subset of an input set, rather than floating point numbers. 
Its running time, when given as a function of the number of primitive operations,  
does not depend on the size of the coefficients.  
Two significant combinatorial algorithms are Clarkson's random sampling algorithm, which uses the solutions of small subproblems 
of size $O(d^2)$, 
and the classic simplex algorithm from linear optimization. 

Clarkson observed that his algorithm could be applied to other problems such as integer programming and 
smallest enclosing ball. Sharir and  Welzl \cite{sw-cblpr-92} introduced an abstract framework, the 
LP-type problems, to describe the requirements for these algorithms to work. 
Other work in this direction applying 
randomization and abstractions of linear programming include \cite{gartner1995subexponential, bsv-dsapg-01,  Halman:2007gu}. 
Last, but not least, in 2008, G\"artner, Matou{\v{s}}ek, R\"ust and \v{S}kovro\v{n}  \cite{ViolatorSpaces2008} established a hierarchy of 
algorithmic frameworks. Starting with the most concrete definition, and progressing to the more abstract case, they propose the 
following three levels of abstraction:
\begin{definition}
\label{concrete}
A \textit{concrete LP-type problem} is a triple $(X,\HH, \preceq)$, where $X$ is a set linearly ordered
by $\preceq$, $\HH$ is a finite multiset whose elements are subsets of $X$, and for any $\G \subseteq \HH$,
 if $\cap \G \neq \emptyset$, then $\cap \G$ has a unique minimum element with respect to $\preceq$. 
We call this unique minimum element $\omega(G)$ (this is in relation to the more abstract notion below).  
\end{definition}
\begin{definition}
\label{def:LP-type}
Consider a pair $(\HH, \omega)$, where $\HH$ is a finite 
set of constraints, 
and $\omega$ is a function 
$\omega:\F \in 2^{\HH} \rightarrow \Lambda$, where $\Lambda$ is a linearly ordered set including
an element $\infty$.
We say $h$ \textit{violates} $F$ when $\omega(F+h) > \omega(F)$. 
The pair $(\HH,\omega)$ is an \emph{LP-type problem} if it satisfies:
\\\begin{tabular}{ll}
{\it Monotonicity}: & $\omega(G + h) \geq \omega(G)$ holds for all $G\subseteq \HH$ and $h \in \HH$, and\\
{\it Locality}: & For all $F \subseteq G \subseteq \HH$, such that $\omega(F) = \omega(G)$, \\
 & $h$ violates $F$ if and only if $h$ violates $G$. 
\end{tabular}
\end{definition}
\begin{definition}
\label{defn:violator}
A {\it violator space} is a pair $(\HH,\vi)$, where $\HH$  is a finite set 
of constraints and $\vi$  a mapping $2^\HH\to2^\HH$ from subsets of 
$\HH$ to their sets of violators, such that the following two axioms hold: 
\\\begin{tabular}{ll}
{\it Consistency}: & $G\cap \vi(G)=\emptyset$ holds for all $G\subseteq \HH$, and\\
{\it Locality}: & $\vi(G)=\vi(F)$ holds for all $F\subseteq G\subseteq \HH$ such that
$G\cap \vi(F)=\emptyset$.\\
\end{tabular}
\end{definition}

The definition of LP-type problems abstracts away the ground set $X$ used in concrete LP-type problem 
and the definition of violator spaces abstracts further, taking away the objective function $\omega$. 
These abstract frameworks are useful because we can often show that a computational challenge has
an algorithmic solution, by proving the axioms are satisfied for an LP-type problem, perhaps without 
showing that its constraints define convex, quasi-convex, or even connected sets. This is often done by 
using a Helly-type theorem, as we see in the next section. 
We summarize here some of the results of G\"artner, Matou{\v{s}}ek, R\"ust and \v{S}kovro\v{n} 
on the relationship between the levels of abstraction:

\begin{theorem}[G\"{a}rtner, Matou\v{s}ek, R\"{u}st,  \v{S}kovro\v{n} \cite{ViolatorSpaces2008}]
i) For every LP-type problem, there is a concrete LP-type problem whose bases achieve the same
values of $\omega$ (so in this sense all LP-type problems are concrete). \\ 
ii) All LP-type problems are violator spaces, 
but not all violator spaces are LP-type problems. 
\end{theorem}

The simplex method pivots from one basic feasible solution to an adjacent one.  
That is, it visits vertices of a polyhedron by moving along its
one-dimensional faces.  
The following abstractions replace these notions.
A \emph{basis} of $G \in 2^{\HH}$ is a minimal cardinality subset $B \subseteq G$ such that the solution of $B$ 
is the solution of $G$. 
For  LP-type problems, we define: 
\begin{definition}
A \emph{basis} 
$B\subseteq \HH$ is a subset such that $\omega(B - h) < \omega(B)$ holds for all $h \in B$. 
For $G\subseteq \HH$, a {basis of $G$}  is a minimal subset $B$ of $G$ with $\omega(B)=\omega(G)$. 
\end{definition}
\begin{definition}
The size of a largest basis of an LP-type problem is called its \emph{combinatorial dimension} 
and denoted by $\delta$.
\end{definition}
For example, the combinatorial dimension of linear programs in $d$ variables is $d+1$.
Similar definitions can be made for violator spaces but we omit them here.  
As with the simplex algorithm,
the solution found by a
combinatorial algorithm for an LP-type problems is a basis for the input set $\HH$ of constraints. 

The following observation links the combinatorial dimension of an LP-type problem with a Helly-type
theorem about its constraint set $\HH$:
\begin{lemma}[Amenta \cite{Amenta:1994gs}]
\label{lem:LP-typeObs}
A concrete LP-type problem of combinatorial dimension $\delta$
is always associated with a Helly-type theorem, in which the Helly number is $\delta+1$: 
a finite set $\F \in 2^{\HH}$ of constraints has non-empty intersection if and only if every subset 
$\G \subseteq \F$ of size  $\delta+1$ has non-empty intersection. 
\end{lemma}

Now let us describe the combinatorial primitives required in these frameworks.
The only combinatorial primitive operation required with violator spaces is the \emph{violator test primitive},
which reports whether $h \in \HH$ violates $G \subset \HH$. 
Simplex-like combinatorial algorithms for  LP-type problems might also use a 
\emph{basis computation primitive},  corresponding to the pivot operation in linear programming:
given a subset $B+h \subset \HH$ of constraints, where $B$ is a basis
and $h$ is a violator of $B$, produce a new basis $B' \subseteq B+h$. 
Basis computations can always be accomplished via violator tests, in time
exponential in $\delta$, but for some LP-type problems they can be done more quickly.  
A simplex algorithm steps from basis to adjacent basis using basis computations; they avoid
cycling. Clarkson's algorithm  \cite{c-lvali-95}, however, can compute a basis for violator
spaces in which a series of basis computation steps might form a cycle.  
Acyclic violator spaces are precisely LP-type problems \cite{ViolatorSpaces2008}.   

Finally, we consider the running times of combinatorial LP algorithms. 
Clarkson 
stated the following result about linear programs and integer linear programs (ILPs):
\begin{theorem}[Clarkson \cite{c-lvali-95}]
Given an $n \times d$ matrix $A$, a vector $\mathbf b \in \R^n$ and the integer program $\min\{c^T \mathbf x: A\mathbf x\leq\mathbf b, \mathbf x\in \Z^d, \mathbf 0\leq \mathbf x \leq\mathbf u\}$,
one can find a solution to this problem in a expected number of steps of order $O(d^2n \log(n)) + d \log(n) O(d^{d/2})$. 
While the algorithm is in  general exponential, it gives the best complexity for solving integer linear programs when the number of variables $d$ is fixed. 
\end{theorem}

More generally, in terms of violator spaces, we have:
\begin{theorem}\cite[Theorem~27]{ViolatorSpaces2008} \label{keytoolvio}
Using Clarkson's algorithm, a basis of $\HH$ in a violator space $(\HH,\vi)$ can be found 
using an expected $O\left(\delta n + \delta^{O(\delta)}\right)$  calls to the violator test primitive. 
\end{theorem}
The Ph.D thesis \cite{SkovronPhDThesis} shows that violator spaces 
give the most general framework in which Clarkson's algorithm converges to a solution
within this time bound.  

Simple analyses \cite{seidel1991small, sw-cblpr-92}
show that the randomized  
dual-simplex algorithm (also known as random facet) 
solves any LP-type problems in time using 
$O(n)$ violator tests and basis computations (where $n$ is the number of constraints), 
with an exponential upper bound on the running time in terms of $\delta$. 
This often gives an efficient algorithm 
when $\delta$ is a constant, to certify whether a family $\F$ of input objects has some Helly-type property.  
A sophisticated analysis of a combined algorithm, using the 
randomized simplex algorithm to solve the small subproblems in Clarkson's algorithm, 
provided the first sub-exponential bound for a combinatorial linear programming algorithm.  We say that the running time is sub-exponential in $\delta$ if 
its running time is exponential in some function in $o(\delta)$, such
as $\sqrt{\delta}$. 

\begin{theorem}\cite{msw-sblp-92,kalai1992subexponential}
The combined
algorithm requires
$O(e^{O(\delta \lg n)})$ basis computations and $O(ne^{O(\delta \lg n)})$ violation tests 
\end{theorem}
When both primitive operations can be performed in sub-exponential time we get an 
overall sub-exponential time algorithm. 


Violator spaces are a powerful general framework. Recently in \cite{algebraicgeo-violatorspace}, the authors prove that problems about
the solvability of systems of non-linear polynomial equations can be approached via violator spaces and Clarkson's algorithm can be adapted to
this situation. 
The Helly numbers in question correspond to those discussed in Theorem \ref{thm:DF-Helly}.

\subsection{Using LP-type problems to prove Helly-type theorems}

As we have seen above, the Helly numbers, in the form of a combinatorial dimension, are important measures
of the scalability of various algorithms (both of LP-type problems and violator spaces). 
Amenta \cite{Amenta:1994gs} gave several constructions for objective functions that could be used
to form LP-type problems from existing Helly-type theorems. 
Often, either lexicographic ordering or a strictly convex function such as distance from the origin 
can be used to construct an operator $\preceq$ in the definition of a concrete LP-type problem, above. 
Using this idea, many of the Helly-type theorems in this survey lead easily to concrete LP-type problems, 
implying that they can be solved in linear time in fixed dimension. 
For instance, consider to Theorem~\ref{mixedint}, which leads to:  

\begin{theorem}
\label{LPmixedInt}
Let $\D$ be a finite multiset of convex sets in $\Z^{d-k} \times \R^k$,  
and let  $\preceq$ be a lexicographic order on  $\Z^{d-k} \times \R^k$. 
Then finding the minimum point of $\cap \D$, if one exists, is a concrete 
LP-type problem of combinatorial dimension at most $\delta = 2^{d-k} (k+1) - 1$. 
\end{theorem}
\begin{proof}
The fact that it is a concrete LP-type problem follows immediately Definition~\ref{concrete}. 
To establish the combinatorial dimension, consider any $\G \subseteq \F$
such that $\cap \G$ is non-empty with unique lexicographic minimum point $x_G \in X$.
Consider any basis $\B \subseteq \G$, and augment $\B$ with the open convex set 
$C_G = \{ x \succ x_G \} \cap \Z^{d-k} \times \R^k$.  
Then $\cap (B \cup C_G) = \emptyset$ while any of its proper subsets
are non-empty, so $|B| \leq 2^{d-k} (k+1)-1$.
\end{proof}

In a similar vein, De Loera et al. showed that every $S$-Helly number, and the corresponding $S$-Helly theorem, allow one to use
Clarkson's algorithm, by directly constructing a violator space.

\begin{theorem}[\cite{de2015calafiore}] 
Let $S \subseteq \real^d$ be a closed set with a finite Helly number $h(S)$. 
Assume that the violation test primitive can be
computed.  Using
Clarkson's algorithm, one can minimize 
$c^T x$, where $c$ is a linear function, 
subject to $f_i(x) 0$, $f_i$ convex constraint for all $i = 1,2,\ldots,n$, and for  $x \in  S$, 
can be found 
using an expected $O\left(\delta n + h(S)^{O(h(S))}\right)$  calls to the violator test primitive. 
Thus the algorithm is linear in $n$ and exponential in $h(S)$. 
\end{theorem}

What is perhaps more surprising is that one can use algorithms that solve LP-type problems to prove Helly-type theorems. 
The following theorem, which we mentioned in Subsection~\ref{topologicalversions} (Theorem \ref{theorem-amenta-disjoint}), 
makes use of Lemma~\ref{lem:LP-typeObs}:

\begin{theorem}[Amenta \cite{Amenta:1996eu}] \label{morris}
Let $(X,\HH,\preceq)$ be a concrete LP-type problem of combinatorial dimension $\delta$ with the 
property that $\preceq$ is a total order on the points of $X$. 
Let $\F$ be a family of subsets of $X$ such that, for every $\G \subseteq \F$ with $\cap \G \neq \emptyset$,
the intersection $\cap \G$ is the disjoint union of at most $r$ elements of $\HH$. 
Then $(X,\F,\preceq)$ is a concrete LP-type problem of combinatorial dimension at most $r(\delta + 1) - 1$.
\end{theorem}
\begin{proof}
We sketch the proof.  Once again, it is clear that $(X,\F,\preceq)$ is a concrete LP-type problem, so we just have to
establish the combinatorial dimension. 
We will count the constraints in 
any basis $B$ such that $\cap B \neq \emptyset$, as follows.   
The point $x$ achieving $\omega(B)$ is contained in one connected component
$C_h$ for each $h \in B$. 
As the unique minimum of $\cap C_h$, it corresponds to a basis consisting 
at most $\delta$ elements of $\HH$.  
We move the elements of $B$ corresponding to those basis elements from $B$ into a holding set $S$. 
Next, we remove the constraint $h \in B$ such that 
$\omega(S+B-h)$ is maximal; $h$ must be unique since the point $x'$
achieving $\omega(S+B-h) \in h'$, for all $h' \in S+B$, but $x' \not \in h$. 
We remove $h$ from $B$. 
The point $x'$ must lie in a different connected component of $S+B-h$ from $x$, since 
$x$ is the minimum point in its connected component and all the constraints of $\HH$ 
determining $x$ are held in $S$. 
Now we repeat; we count the $\leq d$ constraints of $B$ adjacent to $x'$, 
move those to $S$, and then look for another remaining element
of $B$ to remove. 
We continue this process until $B$ is empty; all elements have either been removed or moved into $S$. 
The final set $\cap S$ consists of at most $r$ elements of $\HH$, with the 
minimum of each is determined by at most $d$ constraints, so the size of $S$ is at most $rd$. 
At most $r-1$ elements were removed, each accounting for a new connected component of $\cap S$, 
so the combinatorial dimension is $r(\delta+1) - 1$. 
\end{proof}
Theorem \ref{morris} gives an easy proof of the theorem first conjectured by Gr\"unbaum and Motzkin \cite{Grunbaum:1961fd}
that we mentioned this back in Subsection \ref{topologicalversions}.

\begin{theorem}[Morris and Amenta \cite{Morristhesis},\cite{Amenta:1994gs}]
Let $\F$ be a family of sets in $\R^d$, such that the common intersect of any non-empty finite 
sub-family of $\F$ is the disjoint union of at most $r$ closed convex sets. 
Then $\F$ has a Helly number of at most $r(d+1)$.
\end{theorem} 
We can apply Theorem~\ref{morris} to give new Helly-type results in other situations, for instance with Theorem~\ref{LPmixedInt},  above:
\begin{theorem}
Let $\D$ be the family of convex sets in $\Z^{d-k} \times \R^k$,  and
let  $\preceq$ be a lexicographic order on  $\Z^{d-k} \times \R^k$. 
Let $\F$ be a finite family of subsets of $X$ such that, for every $\G \subseteq \F$ with 
$\cap \G \neq \emptyset$,
the intersection $\cap \G$ is the disjoint union of at most $r$ elements of $\D$. 
Then  $\F$ has Helly number at most $r(2^{d-k} (k+1))$. 
\end{theorem}
 
\begin{openproblem}
Are there S-Helly, integral, or mixed-integer versions of more general Helly-type theorems 
with algorithmic proofs?
\end{openproblem}

\subsection{Other Applications}

In addition to linear programming, there are a few other useful 
sub-exponential convex programming problems in higher dimensional spaces,
which use the dual-simplex algorithm to speed up the solution of the small sub-problems:
\begin{theorem}[G\"{a}rtner \cite{gartner1995subexponential}]
There is a sub-exponential basis computation primitive for the following LP-type problems: 
\begin{itemize}
\item Minimum enclosing ball:  Given $n$ points in $\R^d$, find the ball of minimum radius enclosing 
all the points.
\item Polytope distance: Given two polytopes in $\R^d$, specified by their $n$ vertices, find the distance
between them, or report that they intersect.
\item Maximum margin separator:  Find the linear separator 
between two sets of $n$ points in $\R^d$
that maximizes the distance between the 
separating hyperplane and either set, or report that they cannot be separated. 
This problem is relevant to Support Vector Machine classifiers \cite{balcazar2001provably}, 
and it corresponds to Kirchberger's Theorem.
\end{itemize}
\end{theorem}
All these geometric problems lead to essentially the same computations.

Halman \cite{halman2007simple} showed that several kinds of stochastic games  can be
formulated as LP-type problems and solved in sub-exponential time. In a major breakthrough, 
work based on properties of the Markov decision process has recently led to a new sub-exponential \emph{lower} bound
for the dual-simplex and other randomized combinatorial algorithms for linear programming
\cite{friedmann2014random}; earlier, Matou{\v{s}}ek had shown that there were LP-type 
problems that required expected subexponential time \cite{matouvsek1994lower}.

\begin{openproblem} 
What other natural problems can be solved in expected sub-exponential time by LP-type algorithms? 
\end{openproblem}

Many  geometric problems 
can be solved in time linear in the number of constraints when the 
dimension is fixed.  
Doignon's theorem implies that the combinatorial dimension of integer programming is $2^d-1$. 
Violator tests, or the small sub-problems in Clarkson's algorithm can be solved 
using Lenstra's  algorithm for integer programming 
in fixed dimension \cite{lenstra1983integer} (as pointed out by Clarkson). 
The complexity of Lenstra's algorithm is $O(m s + s^2)$, where $s$ is the bit-complexity of the input and 
$m$ is the number of input constraints;
this is not a combinatorial algorithm, but polynomial in $m$.  
So overall this gives a linear-time algorithm for integer programming in fixed dimension. 
An more efficient algorithm for the small subproblems was 
given by Eisenbrand \cite{eisenbrandfixeddim2}.

Often LP-type problems can also be cast as fixed-dimensional linear or convex programming problems, 
in which $\F$ corresponds to a set of convex constraints and the objective function $\omega$ is a
convex or lexicographic function.  For instance, finding a line transversal of a family 
$K_1,\dots,K_n$ of closed intervals parallel to the $y$-axis in the plane can be  formalized as 
a linear programming problem in the two-dimensional space of lines. 
N. Megiddo generalized this by formulating line transversals of boxes
in any fixed dimension as a constant number of linear programs \cite{megiddo1996program}. 
 
More interesting LP-type problems, corresponding to Helly-type theorems in which the 
sets in $\F$ are not necessarily convex, are often defined using a set-up analogous to quasi-convex programming. 
A good example is:
\begin{itemize}
\item Line transversal of disjoint convex disks in the plane: Given a set $P$ of disk centers the plane, 
with minimum distance one between each pair of points, 
determine the smallest radius $r < 1/2$ such that the family of congruent disks of radius $r$ 
centered at the points of $P$ has a line transversal \cite{Amenta:1994gs}, or report that the 
set has no transversal.  
\end{itemize}
Using the theorem that the Helly number for line transversals of unit disks in the plane is five
\cite{tverberg1989proof}, we find that the combinatorial dimension of this problem is also five, 
even though the dimension of the set of lines in the plane is two, and 
the set of transversals is not necessarily connected. 
Here is a general set-up, in which the constraints grow as a function of a parameter 
$\lambda$.

\begin{definition} 
A \emph{parameterized LP-type problem} is a concrete LP-type problem $(X,\HH,\preceq)$
in which the ground set
$X =  (\lambda', X')$, the linear order $\preceq = \lambda'$, and the level sets are nested.
That is, a level set $h_{\lambda} =  \{  x' :   (\lambda',x') \in \HH \mbox{ and } \lambda' = \lambda \}$
is contained in the level set $h_{\mu}$ for all $\mu \geq \lambda$. 
\end{definition}
\begin{theorem}[\cite{Amenta:1994gs}, Theorem 5.1]
Let $(X,\HH,\preceq)$ be a parameterized LP-type problem. 
When the minimum point $\omega(G)$ in a subset of constraints $G \in H$ always exists and is 
achieved at a unique point, and the Helly number of the level sets $h_\lambda$ for $h \in \HH$ 
is $\delta$, then the combinatorial dimension is also $\delta$.
\end{theorem}
This is a generalization of quasi-convex programming, in which the level sets $h_{\lambda}$ are convex.  
Note that with this construction the Helly number and the combinatorial dimension are identical, 
unlike in the more straightforward definition of concrete LP-type problems in which the constraints themselves
have Helly-type property, rather than their level sets. 
Some useful parameterized LP-type problems include: 
\begin{itemize}
\item Given two convex polygons in the plane, 
find their Hausdorff distance \cite{amenta1994bounded} under translation; this is relevant to 
computer vision \cite{basri1997recognition}.
\item Covering with three squares (rectilinear 3-center problem):
Given a set of $n$ points in the plane, find the smallest three congruent squares that cover all of the 
points \cite{sharir1996rectilinear}.
\item Given a set of $n$ spheres, all contained in the unit ball in $\R^d$, find the M\"{o}bius transformation 
of the unit ball that maximizes the minimum radius of the transformed spheres \cite{bern2001optimal}.  This has
applications to graph layout on the sphere or on the hyperbolic plane. 
\end{itemize}

Finding a line transversal of balls in $\R^3$ can also be done in time linear in the number of constraints by Theorem \ref{cheongetal}.
That result applies to a more general family of sets, so-called 
``pairwise inflatable" balls.  It improves a result of
Holmsen, Katchalski and Lewis \cite{holmsen2003helly}. 
The lower bound for this Helly number is $2d-1$ \cite{cheong2012lower}.
\begin{openproblem}
Is there an LP-type problem of combinatorial dimension $2d-2$ for finding a line transversal of 
disjoint unit balls in $\R^3$, which could be used to improve this Helly number? 
\end{openproblem}

In general, the LP-type problems are associated with ``standard" LP-type problems. 
As we have seen in this survey, there are many related families of combinatorial geometric theorems. 
\begin{openproblem}
Are there non-trivial computational applications of fractional or volumetric Helly-type theorems, 
Hadwiger-type theorems, or
of $(p,q)$ theorems?
\end{openproblem}

Not all computational applications of Helly's theorem involve LP-type algorithms. 
\begin{theorem}[Centerpoint theorem \cite{Neumann-centerpoint, Rado-centerpoint}]
Let $P$ be a set of $n$ points in $\R^d$. 
There exists a point $p \in \R^d$  such that every closed halfspace 
with boundary passing through $p$ contains at least
$\lceil \frac{n}{d+1} \rceil$ of the points in $P$. 
\end{theorem}
Such a point $p$ is called a \emph{centerpoint}, and it can be seen as a generalization of the 
median to data in higher-dimensional Euclidean space. 
The \emph{Tukey depth} of an arbitrary point $p$ is the minimum number of 
points contained in a closed halfspace with boundary passing through $p$; 
a center point has Tukey depth $n/{(d+1)}$, but depending on the distribution 
the Tukey depth of a point might be as high as $n/2$. 
Finding or approximating center points or 
points of maximal Tukey depth is a difficult computational challenge in statistics,
since, naively, the number of halfspaces one needs to consider is $O(n^d)$. 
Using the dependencies between these half spaces leads to a more efficient algorithm
\cite{chan2004optimal} for points of maximal Tukey depth with running time $O(n^{d-1})$.
Finding approximate center points in sub exponential time
\cite{clarkson1996approximating,miller2009approximate,mulzer2013approximating} is very 
useful in partitioning problems for efficient  parallel computation, and of recent interest in statistics \cite{cuesta2008random}. 
The \emph{ham-sandwich theorem}  states that for any $d$ finite sets of points in $\R^d$ there is a hyperplane that bisects all of the sets at once, 
i.e., has at most half of the points on each side. The computational complexity of this problem and its connections to other problems
in convexity is discussed in \cite{apps-sandwich}.

\subsection{Chance-constrained and $S$-optimization} \label{soptimization}



In all the above applications and situations one has a randomized algorithm to solve a deterministic geometric optimization 
problem using discrete data (e.g., finitely many data points).  But, in some situations, the data points are uncertain, governed by a probability distribution, or part of a continuously parametrized set of constraints. Consider for example the following twist on the problem of smallest-radius enclosing ball of finitely many points: we are given points $(u_1,u_2,\dots,u_d) \in \R^d$, 
belonging to an unknown measurable set. Our goal is to find the center $ x$ of a ball of smallest radius $R$  that contains a ``large proportion'' of those points, i.e. we are willing to allow some of the points to be outside the ball.
The data points are uncertain, the bad news is  we may not  know explicitly the probability measure. 
We are trying to make a decision for the center $x$ and the radius $R$, that guarantees a large
proportion of the constraints are satisfied, but we allow some of the constraints not to be satisfied. 
In our example, there is a distance constraint per possible point, but not all the distance conditions have to be satisfied at once, we only expect this with high probability.
The optimization problem is a classical example of a \emph{chance-constrained convex optimization} problem. 

 \begin{equation*}
\begin{split}
     \min \quad & R \\
    \text{subject to} \quad & \operatorname{Pr}\left[\left\{  \sqrt{ \sum_1^d (x_i - u_i)^2 } - R  \leq 0\right\}\right] \geq 1-\epsilon,    x \in \R^{d+1}.
\end{split}
\end{equation*}
\vskip .1cm

In principle there are infinitely many constraints to consider, but it is it possible to get a reasonable solution, with only small proportion of constraints not satisfied,  using partial deterministic convex optimization problem consisting of $N$ samples. In our example, by selecting $N$ out of the given points and carrying a computation for those, we can 
get a candidate center and radius. The induced ball leaves only a small proportion of the points out.

Indeed, and idea presented by Calafiore and Campi  \cite{calafiorecampi2005,calafiorecampi2006} (one where Helly's theorem 
plays a significant role) is to consider only a finite number of sampled instances of the uncertainty affecting the system, and to solve the corresponding standard convex problem.  
For continuous real-valued variables, the number of scenarios $N$ that need be considered
is reasonably small and that the solution of the scenario problem has generalization properties, i.e., it satisfies with high
probability also unseen scenarios. They proved an efficient bound on the sample size $N$ needed for the partial convex optimization problem.
The size of $N$ increases slowly with the required probabilistic levels of success.

More formally, the original  family of convex optimization problems has convex constraints which are parameterized 
$\{ \min  g( { x})   \ \ \text{subject to}  \ f({ x},w) \leq 0, \ \ w \in \Delta\}.$
Here $g(x)$ is a convex objective function and $f$ is convex function in the variables $ x$ and parameterized 
using a measurable set $\Delta$ from which
the variables $w$ take their values.  Given a  {\em tolerance of risk} $\epsilon$ one wishes to find an optimal solution for

\begin{equation*}
\begin{split}
 CCP(\epsilon) =\min \quad & g( { x}) \\
    \text{subject to} \quad & Pr[\{ w :  f(x,w) \leq 0\}] \geq 1-\epsilon, \\
	& x \in K \, \text{convex set}, \\
    &{x} \in \R^d.
\end{split}
\end{equation*}

The goal is to find a solution that is optimal among all solutions that satisfy all but ``a few constraints''. In the Calafiori and Campi algorithm, one collects $N$ random 
independent identically distributed  samples $w^1, \dots, w^N$, and constructs a \emph{sampled convex program}

\begin{align*}
   SCP(N)= \min \quad & g( { x}) \\
    \text{subject to} \quad & f( x, w^i) \leq 0 ,\quad i=1,2,\ldots, N, \\
	&  x \in K \text{ convex set}, \\
	&  x \in \R^d.
\end{align*}

This leads to two natural questions.
\begin{itemize}
	\item How is the solution of $SCP(N)$ related to the solution for $CCP(\epsilon)$?
	\item How many samples need to be drawn in order to guarantee that the resulting randomized optimal solution from $SCP(N)$
violates only a { ``small portion''} of the full set of constraints?
\end{itemize}

Given  $x \in \R^d$, the \emph{probability of violation} of $x$ is $V (x):=Pr[\{w \in \Delta : f(x,w)>0 \}].$
A solution $x$ with small associated $V (x)$ is feasible for  ``most'' of the problem instances.
For example, if we take the uniform probability density  (with respect to Lebesgue's measure), then $V (x)$ is just  the volume of the set of parameters $w$ for which $f(x,w) \leq 0$  is violated. 
Given $\epsilon \in [0,1]$, if for $\bf x$ we have $\operatorname{Pr}[f( x, w) \leq 0] \geq 1-\epsilon$,  
we say it is an \emph{$\epsilon$-level feasible solution}.  This is precisely when $V(x)<\epsilon.$
If  $x^N$ is the optimal solution of $SCP(N)$.  $V(x^N)$ is a random variable in the space $\Delta^N$ with 
product probability measure $Pr\times Pr \times \dots \times Pr=Pr^N$.

\begin{theorem}[Feasibility of $x^N$] \label{originalCC}
Let  $0< \epsilon \leq 1$ (tolerance) , $0<\delta<1$ (lack of confidence). Then, either $CCP(\epsilon)$ is infeasible 
or when $N$ is sufficiently large. With probability at least $1-\delta$, the optimal solution $x^N$ of $SCP(N)$ is a 
feasible solution of $CCP(\epsilon)$ with $V(x^N) \leq \epsilon$. Here sufficiently large means:
$$N \geq \frac{2d}{\epsilon} \ln\left(\frac{1}{\epsilon}\right) + \frac{2}{\epsilon} \ln\left(\frac{1}{\delta}\right)+2d$$
\end{theorem}

\begin{theorem}[Optimality gap of $x^N$]
Let   $0< \epsilon \leq 1$ (tolerance) , $0<\delta<1$ (lack of confidence).  Define $\epsilon_1=1- (1-\delta)^{1/N}$.  Then  
\begin{enumerate}

\item Let $J^{\epsilon}$ be the optimal objective value of $CCP(\epsilon)$ and $J^N$ the optimal value of $SCP(N)$.
With probability at least $1-\delta$, if $SCP(N)$ is feasible, it holds that $J^\epsilon \leq J^N$.
 
\item $\epsilon_1 < \epsilon$ and if $CCP(\epsilon_1)$ is feasible, then with probability at least $1-\delta$ we have $J^N \leq J^{\epsilon_1}$. 

\end{enumerate}
\end{theorem}

As it turns out, it was proved in \cite{de2015calafiore} that the results of Calafiore and Campi for real-valued variables can be extended in greater generality.
In Section \ref{Shelly} we have discussed $S$-Helly theorems for different proper subsets of $\R^d$.
The authors of \cite{de2015calafiore} introduced the notion of \emph{$S$-optimization}, a natural  generalization of continuous, integer, and mixed-integer optimization:

\begin{definition} 
Given $S \subset \R^d$, the optimization problem with equations and inequalities and variables taking values on $S$,
\begin{align*}
    \max \quad              & f( x)  \\
    \text{subject to} \quad & g_i( x) \leq 0, \quad i=1,2,\ldots, n, \\
                    & h_j( x) = 0,\quad j=1,2,\ldots, m,\\                  &{ x} \in S,
\end{align*}
will be called an {\em $S$-optimization problem.} 
\end{definition}

Clearly when $S=\R^d$ the $S$-optimization problem is the usual continuous optimization problem, 
 $S=\Z^d$ is just integer optimization, and $S=\Z^k \times \R^{d-k}$ is the case of mixed-integer optimization. 
When only linear constraints are present this is an {\em $S$-linear program.} When all constraints are convex we call
this an \emph{$S$-convex program}. 

To motivate the study of $S$-optimization for unusual sets $S$, below is an example that shows 
the  modeling power of using sophisticated  $S \subset \R^d$ (typically discrete sets). 

\begin{example}
Given a graph $G=(V,E)$, we reformulate the classic graph $K$-coloring query as the solvability of the following linear system of modular inequations:
For all $(i,j)$ in $E(G)$ consider the inequations $c_i \not \equiv c_j \mod K$. This is a system on $\abs{V}$ variables and it has a solution if and only if the graph is $K$-colorable.
Note that the set of points $ c=(c_1,\dots,c_{\abs{V}})$ with $c_i \equiv c_j \mod K$ is a lattice, which we call $L_{i,j}$. Therefore, solving our system of inequalities is equivalent to finding a $ c\in S=\Z^{\abs{V}}\setminus(\bigcup_{i,j} L_{i,j})$. Consequently, the problem of deciding $k$-colorability is equivalent to the problem of finding a solution to an $S$-linear system of equations, where the variables take values on $S$,  the set difference of a lattice and a union of several sublattices.
\end{example}

\begin{theorem} \label{gen-calafiorecampi} Let $S \subseteq \R^d$ be a set with a finite Helly number $h(S)$. Let $0< \epsilon \leq 1$ (tolerance), $0<\delta<1$ (distrust) be chosen parameters.
Let $f({ x},w)$ be a convex function in $ x$ and  measurable in $w$. Suppose there is an  optimal value $ x_*$ of the linear minimization chance-constrained problem 

\begin{equation*}
\begin{split}
  CCP(\epsilon)=  \min \quad & c^T  x \\
     \text{subject to} \quad & Pr[f( x, w) > 0] < \epsilon, \\
	& { x} \in K \, \text{convex set}, \\
    &{ x} \in S.
\end{split}
\end{equation*}

Then from a sufficiently large random sample of  $N$ different $i.i.d$ values for $w$ (specifically, $w^1, w^2,\dots,w^N$), $ x_*$ 
can be $\delta$-approximated by $ x_N$,  the optimal solution of the convex optimization problem 

\begin{align*}
  SCP(N)=  \min \quad & c^T  x \\
    \text{subject to} \quad & f( x, w^i) \leq 0 ,\quad i=1,2,\ldots, N, \\
	&  x \in K \text{ convex set}, \\
	&  x \in S.
\end{align*}
More precisely, if $x_N$ exists and  the size of the sample $N \geq \frac{2(h(S)-1)}{\epsilon} \ln(1/\epsilon) + \frac{2}{\epsilon} \ln(1/\delta)+2(h(S)-1),$ then
the undesirable event of high-infeasibility $V(x_N)> \epsilon$ has probability less than $\delta$ of occurring. 
\end{theorem}

Recently in \cite{Gartner15}, G\"artner investigated how to use LP-type problems and violators spaces in the context of chance-constrained optimization.
Given a nondegenerate LP-type problem of combinatorial dimension $\delta$, he considered sampling, with the subsequent removal of the $k$ constraints that lead to 
the best improvement of objective function. 

To conclude we wish to discuss yet another way in which Helly numbers and the optimal Helly configurations (i.e., families that show the Helly number is smallest possible already)
appear in deterministic algorithms for mixed integer optimization. 

The theory presented in Subsection \ref{Shelly} directly influences the modern methods for solving mixed integer programming problems.
The most popular  technique to solve mixed-integer optimization problems
is \emph{branch-and-cut} (see the book \cite{confortietalbook} for an introduction). Branch-and-cut is based on three main ingredients: relaxations
of the problem (typically to continuous linear optimization), clever enumeration, and an appropriate tightening of the
formulation by adding cuts. Namely, tightening the formulation is done by adding new valid inequalities called \emph{cutting planes}. Commercial software uses a variety of cutting plane techniques, but
the most successful is the Gomory mixed-integer cuts introduced in  \cite{gomoryorigin}. 

Extending Gomory's technique has led to a large variety of work which relates to $S$-optimization
and the construction of \emph{maximal $S$-free sets}.  These are the $S$-facet polyhedra from Lemma \ref{lem:hollow}. In 2007,  Andersen, Louveaux, Weismantel and Wolsey (see \cite{andersenetal} for  the fundamental association to geometry)  put forward a new idea for the generation of cutting planes.  Essentially, they showed that, besides non-negativity constraints, the facet-defining inequalities are associated with \emph{splits} (a region between two parallel lines), triangles and quadrilaterals whose interior does not contain an integer point. This allows one to derive valid inequalities by exploiting the combined effect of two rows, instead of a single row. The extension of this model to any dimension (i.e., number of integer variables) was pioneered by Borozan and Cornu\'ejols \cite{borozancornuejols} and by Basu, Conforti, Cornu\'ejols and Zambelli \cite{basuetal}.
As described in Subsection \ref{Shelly}, the maximal $S$-free sets can be used to compute the $S$-Helly numbers, which bring a connection to integer optimization.

Once more one gets the striking result that the  facet-defining inequalities of an integer opimization problem are naturally associated with full-dimensional convex sets whose interior do not contain an integer point. Furthermore these sets are polyhedra.  Lov\'asz stated  earlier in \cite{lovaszlatticefree} that maximal convex sets whose interior does not contain an integer point are polyhedra, but the first complete proof appears in \cite{basuetal}, and an alternate proof can be found in \cite{averkovproofoflovasz}.

These proofs use the simultaneous approximation theorem of Dirichlet and Minkowski's convex body theorem. Important results from number theory, convex geometry and Helly numbers are  essential to get a computable formula for the cut-generating function. The details are explained in the recent excellent survey by \cite{basuetalsurvey}. This application to cut generation methods has stimulated several direct progress in the combinatorial geometry around Helly's theorem (see e.g., recent work in \cite{de2015helly}). Once more geometry comes to the help of
optimization.

\section*{Acknowlegements}
We thank the participants of the Mathematical Research Community  ``Algebraic and 
Geometric Methods in Discrete Applied Mathematics'' for their enthusiasm and support for this 
book project. The second author is grateful for the support received through an NSA grant. 
The first author gratefully acknowledges the support of NSF-IIS-1117663 and NSF-IIS-0964357. 

\bibliographystyle{amsalpha}

\bibliography{referencessecond.bib}

\end{document}